\documentclass[10pt,a4paper]{article}

\usepackage{latexsym,amsfonts,amsmath,amssymb,amsthm,mathrsfs,color,graphicx,url}

\newtheorem{theorem}{Theorem}
\newtheorem{lemma}{Lemma}
\newtheorem{algorithm}{Algorithm}
\newtheorem{remark}{Remark}
\newtheorem{proposition}{Proposition}
\newtheorem{definition}{Definition}
\newtheorem{corollary}{Corollary}

\newcommand{\rd}{\mathrm{d}}
\newcommand{\rtr}{\mathrm{tr}}
\newcommand{\bsa}{\boldsymbol{a}}
\newcommand{\bsc}{\boldsymbol{c}}
\newcommand{\bsk}{\boldsymbol{k}}

\newcommand{\bsq}{\boldsymbol{q}}
\newcommand{\bsu}{\boldsymbol{u}}
\newcommand{\bsx}{\boldsymbol{x}}
\newcommand{\bsy}{\boldsymbol{y}}

\newcommand{\bsrho}{\boldsymbol{\rho}}

\newcommand{\bszero}{\boldsymbol{0}}

\newcommand{\NN}{\mathbb{N}}
\newcommand{\RR}{\mathbb{R}}
\newcommand{\ZZ}{\mathbb{Z}}
\newcommand{\Dcal}{\mathcal{D}}
\newcommand{\Ecal}{\mathcal{E}}
\newcommand{\Fcal}{\mathcal{F}}
\newcommand{\wal}{\mathrm{wal}}
\newcommand{\wor}{\mathrm{wor}}
\newcommand{\vol}{\mathrm{vol}}

\allowdisplaybreaks

\begin{document}

\title{Construction of interlaced polynomial lattice rules for infinitely differentiable functions}

\author{Josef Dick\thanks{School of Mathematics and Statistics, The University of New South Wales, Sydney, NSW 2052, Australia (\tt{josef.dick@unsw.edu.au})}, Takashi Goda\thanks{Graduate School of Engineering, The University of Tokyo, 7-3-1 Hongo, Bunkyo-ku, Tokyo 113-8656, Japan (\tt{goda@frcer.t.u-tokyo.ac.jp})}, Kosuke Suzuki\thanks{School of Mathematics and Statistics, The University of New South Wales, Sydney, NSW 2052, Australia (\tt{kosuke.suzuki1@unsw.edu.au})}, Takehito Yoshiki\thanks{School of Mathematics and Statistics, The University of New South Wales, Sydney, NSW 2052, Australia (\tt{takehito.yoshiki1@unsw.edu.au})}}

\date{\today}

\maketitle

\begin{abstract}
We study multivariate integration over the $s$-dimensional unit cube in a weighted space of infinitely differentiable functions. It is known from a recent result by Suzuki that there exists a good quasi-Monte Carlo (QMC) rule which achieves a super-polynomial convergence of the worst-case error in this function space, and moreover, that this convergence behavior is independent of the dimension under a certain condition on the weights.

In this paper we provide a constructive approach to finding a good QMC rule achieving such a dimension-independent super-polynomial convergence of the worst-case error. Specifically, we prove that interlaced polynomial lattice rules, with an interlacing factor chosen properly depending on the number of points $N$ and the weights, can be constructed using a fast component-by-component algorithm in at most $O(sN(\log N)^2)$ arithmetic operations to achieve a dimension-independent super-polynomial convergence. The key idea for the proof of the worst-case error bound is to use a variant of Jensen's inequality with a purposely-designed concave function.
\end{abstract}
\emph{Keywords}:\; Quasi-Monte Carlo integration, super-polynomial convergence, interlaced polynomial lattice rules, infinitely differentiable functions\\
\emph{MSC classifications}:\; 65C05, 65D30, 65D32

%%%%%%%%%%%%%%%%%%%%%%%%%%%%%%%%%%%%%%%%%%%%%%%%%%%%%%%%%%%
%%%%%%%%%%%%%%%%%%%%%%%%%%%%%%%%%%%%%%%%%%%%%%%%%%%%%%%%%%%
%%%%%%%%%%%%%%%%%%%%%%%%%%%%%%%%%%%%%%%%%%%%%%%%%%%%%%%%%%%
\section{Introduction}\label{sec:intro}
We study the approximation of multivariate integrals of real-valued functions defined over the $s$-dimensional unit cube $[0,1]^s$,
\begin{align*}
 I(f) = \int_{[0,1]^s}f(\bsx) \rd \bsx .
\end{align*}
Quasi-Monte Carlo (QMC) integration approximates $I(f)$ by using a deterministically chosen finite point set $P\subset [0,1]^s$ as
\begin{align*}
 I(f;P) = \frac{1}{|P|}\sum_{\bsx\in P}f(\bsx) ,
\end{align*}
where $|P|$ denotes the cardinality of $P$. Note that we interpret $P$ here as a set in which the multiplicity of elements matters. In order to make the integration error $|I(f;P)-I(f)|$ small for a class of functions $f$, $P$ needs to be carefully designed depending on the class to which the function $f$ belongs. Digital nets and sequences are a well-known choice for constructing good quadrature points for several classes of functions \cite{DPbook,Niedbook}.

A classical criterion for measuring the distribution properties of point sets is the so-called star-discrepancy. The Koksma--Hlawka inequality bounds the integration error using a point set by the star-discrepancy of this point set times the total variation in the sense of Hardy and Krause, see for instance \cite[Chapter~2, Section~5]{KNbook}. Thus a low-discrepancy point set of $N$ points yields a small integration error bound, typically of order $N^{-1+\varepsilon}$ with arbitrarily small $\varepsilon>0$, assuming that the function $f$ has bounded total variation in the sense of Hardy and Krause. Regarding explicit constructions of low-discrepancy digital nets and sequences, we refer to \cite[Chapter~8]{DPbook} and \cite[Chapter~4]{Niedbook}. Polynomial lattice point sets, first introduced in \cite{Nied92}, are a special construction method for digital nets and have been extensively studied in the literature, see for instance \cite[Chapter~10]{DPbook} and \cite{Pill12}. Polynomial lattice rules are QMC rules using a polynomial lattice point set as quadrature points. While we usually resort to some computer search algorithm to find good polynomial lattice rules for $s>2$, the major advantage of polynomial lattice rules lies in their flexibility, that is, we can design a suitable QMC rule for the problem at hand.

In order to achieve a faster convergence of the integration error, explicit constructions of point sets, referred to as \emph{higher order digital nets}, have been established by Dick \cite{Dick07,Dick08} which can fully exploit the smoothness of an integrand. Specifically QMC rules using higher order digital nets achieve the optimal convergence rate of the integration error of order $N^{-\alpha+\varepsilon}$ with arbitrarily small $\varepsilon>0$, when the function $f$ has square integrable partial mixed derivatives up to order $\alpha \ge 2$ in each variable. We remark that recent applications in the area of uncertainty quantification, in particular partial differential equations with random coefficients, are in need of using these types of quadrature rules, see for instance \cite{DKLNS14}. The above result by Dick is based chiefly on analyzing the decay of the Walsh coefficients of smooth functions \cite{Dick08,Dick09}.

Numerical integration of infinitely many times differentiable functions in certain function spaces has recently  been considered in \cite{DLPW11, DG15, IKLP15, KPW14}. However, the results on higher order digital nets in \cite{Dick07, Dick08} do not improve if one assumes that the integrand is infinitely many times differentiable. More precisely, if one sets $\alpha = \infty$ in \cite{Dick07, Dick08} one obtains constants which are infinite and the error bounds become trivial. To improve the error bounds in these papers for function spaces consisting of infinitely many times differentiable functions using higher order digital nets requires new bounds on the Walsh coefficients. Such an analysis of the Walsh coefficients was recently done in \cite{SY15,Yoshiki15}, where they obtained a space $\Fcal_s$ of infinitely differentiable functions whose Walsh coefficients decay with a certain order. The worst-case error in $\Fcal_s$ by a digital net is closely related to the Walsh figure of merit (WAFOM) introduced in \cite{MSM14,Suzuki15}, which is one of the computable quality criteria of digital nets, although WAFOM was originally derived in a different way from \cite{SY15,Yoshiki15}. Moreover, Suzuki \cite{Suzuki15_2} considered a \emph{weighted} space $\Fcal_{s,\bsu}$ of infinitely differentiable functions and studied tractability of multivariate integration in $\Fcal_{s,\bsu}$, where the positive real numbers $\bsu=(u_j)_{j\in \NN}$ are the weights. His result can be summarized as follows: There exists a good QMC rule 
using a digital net which achieves an \emph{super-polynomial convergence} of the worst-case error in $\Fcal_{s,\bsu}$ as $C(s)e^{-c(s)(\log N)^2}$, and moreover, the convergence can be independent of the dimension $s$ as $Ce^{-c(\log N)^p}$ for some $1<p<2$ under a certain condition on the weights $\bsu$. %Here we mean by accelerating convergence that there exists a constant $w\in (0,1)$ and functions $C,C_1: \NN\to (0,\infty)$ and $p:\NN \to (1,\infty)$ such that the integration error decays as $C(s)w^{(\log N/C_1(s))^{p(s)}}$. Assuming $p$ is greater than 1, the error asymptotically converges faster than any polynomial of $N$. This is why we call this behavior accelerating convergence.

In this paper, beyond the existence result of \cite{Suzuki15_2}, we provide a constructive approach to finding good QMC rules achieving a dimension-independent super-polynomial convergence of the worst-case error. Specifically we prove that \emph{interlaced polynomial lattice rules} can be constructed using a fast component-by-component (CBC) algorithm, in at most $O(sN(\log N)^2)$ arithmetic operations, to achieve a dimension-independent super-polynomial convergence. As first studied in \cite{Goda15,Goda15_2,GD15}, interlaced polynomial lattice rules belong to the family of higher order digital nets and therefore achieve a higher order polynomial convergence of the integration error. We use them as QMC rules achieving a super-polynomial convergence in this paper. For this purpose, we are required to choose an interlacing factor depending on the number of points and the weights, instead of keeping it fixed (as for instance in \cite{Dick07, Dick08}). Furthermore, in order to show the worst-case error bound with a super-polynomial convergence, we purposely design a concave function to modify Jensen's inequality which has been often used in the literature to obtain error bounds with an improved rate of convergence.

Our approach requires to set the weights for constructing a tailored QMC rule, as often encountered in this type of construction algorithms. In practical applications, however, it is not always the case where one can know in advance to which function class the functions of interest belong. To work around this drawback, it must be interesting to study whether a good convergence property which such a tailored QMC rule holds for a specific function class can be also established for other function classes, as discussed for instance in \cite[Remark~1]{GD15}.
We observe in Section~\ref{sec:exp} that our constructed rules empirically work even for some functions not belonging to the target space.
In another direction for constructing a robust QMC rule working for many different function classes, one can implement a more elaborate construction algorithm as given in \cite{Dick12}. However, theoretical analysis of these issues is beyond the scope of this paper and we leave them open for further research.

The remainder of this paper is organized as follows. In the next section, we introduce the necessary background and notation, namely Walsh functions, a weighted space $\Fcal_{s,\bsu}$ of infinitely differentiable functions, our considering super-polynomial convergence and interlaced polynomial lattice rules. We also describe the main results of this paper. Namely, we introduce a component-by-component algorithm, state a result on the convergence behavior of interlaced polynomial lattice rules and discuss the dependence of the worst-case error bound on the dimension. In Section~\ref{sec:error}, we study the worst-case error in $\Fcal_{s,\bsu}$ for QMC rules using a digital net and derive a computable upper bound. We prove in Section~\ref{sec:cbc} that the CBC algorithm can be used to obtain good interlaced polynomial lattice rules which achieve a dimension-independent super-polynomial convergence of the worst-case error. Thereafter we describe the fast CBC algorithm using the fast Fourier transform as in \cite{NC06,NC06_2}, 
and show that interlaced polynomial lattice rules achieving a dimension-independent super-polynomial convergence can be constructed in at most $O(sN(\log N)^2)$ arithmetic operations using $O(N)$ memory. Finally, we conclude this paper with numerical experiments in Section~\ref{sec:exp}.

%%%%%%%%%%%%%%%%%%%%%%%%%%%%%%%%%%%%%%%%%%%%%%%%%%%%%%%%%%%
%%%%%%%%%%%%%%%%%%%%%%%%%%%%%%%%%%%%%%%%%%%%%%%%%%%%%%%%%%%
%%%%%%%%%%%%%%%%%%%%%%%%%%%%%%%%%%%%%%%%%%%%%%%%%%%%%%%%%%%
\section{Background, notation and results}\label{sec:background}
Throughout this paper, we use the following notation. Let $\NN$ be the set of positive integers and let $\NN_0:=\NN\cup \{0\}$. For a positive integer $b\ge 2$, let $\ZZ_b$ be a finite ring with $b$ elements, which we identify with the set $\{0,1,\dots,b-1\}$ equipped with addition and multiplication modulo $b$. For $x\in [0,1)$, we denote its $b$-adic expansion by $x=\sum_{i=1}^{\infty}\xi_i b^{-i}$ with $\xi_i\in \ZZ_b$ for all $i$, which is unique in the sense that infinitely many $\xi_i$ are different from $b-1$. The operators $\oplus$ and $\ominus$ denote digitwise addition and subtraction modulo $b$, respectively. That is, for $x, x'\in [0,1)$ whose unique $b$-adic expansions are $x=\sum_{i=1}^{\infty}\xi_i b^{-i}$ and $x'=\sum_{i=1}^{\infty}\xi'_i b^{-i}$, $\oplus$ and $\ominus$ are defined as
  \begin{align*}
    x\oplus x' = \sum_{i=1}^{\infty}\eta_i b^{-i}\quad \text{and}\quad x\ominus x' = \sum_{i=1}^{\infty}\eta'_i b^{-i},
  \end{align*}
where $\eta_i=\xi_i+\xi'_i \pmod b$ and $\eta'_i=\xi_i-\xi'_i \pmod b$, respectively. Similarly, we define digitwise addition and subtraction for non-negative integers based on their $b$-adic expansions. In case of vectors in $[0,1)^s$ or $\NN_0^s$, the operators $\oplus$ and $\ominus$ are applied componentwise.

%%%%%%%%%%%%%%%%%%%%%%%%%%%%%%%%%%%%%%%%%%%%%%%%%%%%%%%%%%%
\subsection{Walsh functions}\label{subsec:wal}
Walsh functions were first introduced in \cite{Walsh23} for the case $b=2$ and were later generalized to arbitrary base $b  \ge 2$, see for instance \cite{Chrestenson55}. We refer to \cite[Appendix~A]{DPbook} for more information on Walsh functions in the context of numerical integration. We first give the definition for the one-dimensional case.

\begin{definition}\label{def:wal_1}
Let $b\ge 2$ be a positive integer and let $\omega_b:=\exp(2\pi \sqrt{-1}/b)$ be a $b$-th root of unity. We denote the $b$-adic expansion of $k\in \NN_0$ by $k = \kappa_0+\kappa_1b+\dots+\kappa_{a-1}b^{a-1}$ with $\kappa_i\in \ZZ_b$. The $k$-th $b$-adic Walsh function ${}_b\wal_k: [0,1)\to \{1,\omega_b,\dots,\omega_b^{b-1}\}$ is defined as
  \begin{align*}
    {}_b\wal_k(x) := \omega_b^{\kappa_0\xi_1+\dots+\kappa_{a-1}\xi_a} ,
  \end{align*}
for $x\in [0,1)$ with its unique $b$-adic expansion $x=\xi_1b^{-1}+\xi_2b^{-2}+\cdots$.
\end{definition}
\noindent This definition can be generalized to the higher-dimensional case.

\begin{definition}\label{def:wal_s}
Let $b\ge 2$ be a positive integer. For a dimension $s\in \NN$, let $\bsx=(x_1,\ldots, x_s)\in [0,1)^s$ and $\bsk=(k_1,\ldots, k_s)\in \NN_0^s$. The $\bsk$-th $b$-adic Walsh function ${}_b\wal_{\bsk}: [0,1)^s \to \{1,\omega_b,\ldots, \omega_b^{b-1}\}$ is defined as
  \begin{align*}
    {}_b\wal_{\bsk}(\bsx) := \prod_{j=1}^s {}_b\wal_{k_j}(x_j) .
  \end{align*}
\end{definition}

Since we always use Walsh functions in a fixed base $b$, we omit the subscript and simply write $\wal_k$ or $\wal_{\bsk}$ in the remainder of this paper. From the fact that the system $\{\wal_{\bsk}: \bsk\in \NN_0^s\}$ is a complete orthonormal system in $L_2([0,1]^s)$ for any $s\in \NN$ \cite[Theorem~A.11]{DPbook}, we have a Walsh series expansion for any $f\in L_2([0,1]^s)$
  \begin{align*}
    \sum_{\bsk\in \NN_0^s}\hat{f}(\bsk)\wal_{\bsk} ,
  \end{align*}
where $\hat{f}(\bsk)$ denotes the $\bsk$-th Walsh coefficient of $f$, which is defined as
  \begin{align*}
    \hat{f}(\bsk) := \int_{[0,1]^s}f(\bsx)\overline{\wal_{\bsk}(\bsx)}\rd \bsx .
  \end{align*}
For continuous functions $f: [0,1]^s\to \RR$ for which $\sum_{\bsk\in \NN_0^s}|\hat{f}(\bsk)|<\infty$, the Walsh series of $f$ converges to $f$ pointwise absolutely. In fact, for any function $f : [0,1]^s \to \RR$ in a weighted space $\Fcal_{s,\bsu}$ which we consider in this paper, its Walsh series converges to $f$ pointwise absolutely.

%%%%%%%%%%%%%%%%%%%%%%%%%%%%%%%%%%%%%%%%%%%%%%%%%%%%%%%%%%%
\subsection{Weighted function space $\Fcal_{s,\bsu}$}\label{subsec:wal_space}
We first define the function $\mu(a;\cdot) : \NN_0\to \RR$ for a real number $a$. 

\begin{definition}\label{def:weight_1}
Let $a$ be a real number. For $k\in \NN$, we denote its $b$-adic expansion by $k = \kappa_1 b^{c_1-1}+\kappa_2 b^{c_2-1}+\cdots +\kappa_v b^{c_v-1}$ such that $\kappa_1,\dots,\kappa_v \in \{1,2,\dots,b-1\}$ and $c_1>\ldots > c_v >0$. The function $\mu(a;\cdot) : \NN_0\to \RR$ is defined as 
  \begin{align}\label{eq:weight_1}
    \mu(a;k) := \sum_{i=1}^{v}(c_i+a) ,
  \end{align}
and $\mu(a;0):=0$.
\end{definition}

\begin{remark}
Let us consider the case $a=0$. If the sum on the right-hand side of (\ref{eq:weight_1}) which runs over $i=1,\dots,v$ is replaced by the sum which runs over $i=1,\dots,\min(\alpha,v)$ for a fixed $\alpha\in \NN$, we recover the definitions by Niederreiter, Rosenbloom and Tsfasman in \cite{Nied86,RT97} for $\alpha=1$ and by Dick in \cite{Dick08} for $\alpha\ge 2$. Our function $\mu_a$ with $a=0$ has been used in \cite{MSM14,Suzuki15}. The parameter $a$ was included in the definition originally by Yoshiki \cite{Yoshiki15} for $a=1$ and later by Suzuki \cite{Suzuki15_2} for an arbitrary real number $a$.
\end{remark}
\noindent
For the higher-dimensional case, we consider a vector of $s$ real numbers $\bsa=(a_1,\dots,a_s)$ and define the function $\mu(\bsa;\cdot): \NN_0^s\to \RR$ as follows.

\begin{definition}\label{def:weight_s}
Let $\bsa=(a_1,\dots,a_s)$ be a vector of $s$ real numbers, and let $\bsk=(k_1,\dots,k_s)\in \NN_0^s$. The  function $\mu(\bsa;\cdot): \NN_0^s\to \RR$ is defined as
  \begin{align*}
    \mu(\bsa;\bsk) := \sum_{j=1}^{s}\mu(a_j;k_j) .
  \end{align*}
\end{definition}

We are now ready to introduce a weighted space $\Fcal_{s,\bsu}$ of infinitely differentiable functions. Let $\bsu=(u_j)_{j\in \NN}$ be a sequence of positive real numbers which we call \emph{weights}, and we assume that $u_1\ge u_2\ge \cdots >0$ throughout this paper.

\begin{definition}\label{def:function_space}
Let $\bsu=(u_j)_{j\in \NN}$ be a sequence of weights. We define a weighted space $\Fcal_{s,\bsu}$ as
  \begin{align*}
    F_{s,\bsu} := \left\{ f\in C^{\infty}([0,1]^s) : \lVert f\rVert_{F_{s,\bsu}} := \sup_{(\alpha_1,\ldots,\alpha_s)\in \NN_0^s}\frac{\lVert f^{(\alpha_1,\ldots,\alpha_s)}\rVert_{L^1}}{\prod_{j=1}^{s}u_j^{\alpha_j}}<\infty \right\} ,
  \end{align*}
where $f^{(\alpha_1,\ldots,\alpha_s)}$ denotes the $(\alpha_1,\ldots,\alpha_s)$-th mixed partial derivative of $f$, i.e., $(\partial/\partial x_1)^{\alpha_1} \cdots (\partial/\partial x_s)^{\alpha_s}f$.
\end{definition}

In the function space $F_{s,\bsu}$, $u_j$ small means that higher order partial mixed derivatives associated with the $j$-th coordinate must be relatively small. Thus, the weights $\bsu$ play a role in moderating the importance of different variables. Owing to the refined analyses of the Walsh coefficients in \cite{SY15,Yoshiki15}, it was shown that the Walsh coefficients of any function in $\Fcal_{s,\bsu}$ decay with a certain order, as we describe in the following. Let
  \begin{align*}
    m_b := \min_{c=1,2,\dots,b-1}|1-\overline{\omega_b}^{c}| = 2\sin(\pi/b), 
  \end{align*}
and
  \begin{align*}
    M_b := \max_{c=1,2,\dots,b-1}|1-\overline{\omega_b}^{c}| = \begin{cases}
    2 & \text{if $b$ is even,} \\
    2\sin((b+1)\pi /2b) & \text{if $b$ is odd.}
    \end{cases}
  \end{align*}
Moreover, let
  \begin{align*}
    C_b = \begin{cases}
    2 & \text{if $b=2$,} \\
    M_b+\frac{bm_b}{b-M_b} & \text{if $b\ne 2$.}
    \end{cases}
  \end{align*}
Then we have the following.

\begin{proposition}[\cite{SY15,Yoshiki15}]\label{prop:walsh_decay}
Let $\bsu=(u_j)_{j\in \NN}$ be a sequence of weights, and let $m_b$ and $C_b$ be constants depending only on $b$ as above. For any function $f$ in $\Fcal_{s,\bsu}$ and $\bsk\in \NN_0^s$, the $\bsk$-th Walsh coefficient of $f$ is bounded by
  \begin{align*}
   |\hat{f}(\bsk)| \le \lVert f\rVert_{F_{s,\bsu}} b^{-\mu(\bsa;\bsk)} ,
  \end{align*}
where $\bsa=(a_j)_{j\in \NN}$ is a sequence given by $a_j = -\log_b( C_b m_b^{-1} u_j)$ for all $j = 1, \ldots, s$.
\end{proposition}

%%%%%%%%%%%%%%%%%%%%%%%%%%%%%%%%%%%%%%%%%%%%%%%%%%%%%%%%%%%
\subsection{Super-polynomial convergence}\label{subsec:ac_conv}
From \cite{Suzuki15_2}, it is known that there exists a good QMC rule which achieves a dimension-independent super-polynomial convergence of the worst-case error in $\Fcal_{s,\bsu}$ under a certain condition on the weights $\bsu$. Here we briefly recall the result of \cite{Suzuki15_2}.

The initial error in $\Fcal_{s,\bsu}$ is given by the error of the zero algorithm, i.e.,
\begin{align*}
  e^{\wor}(\Fcal_{s,\bsu};\emptyset) = \sup_{\substack{f\in \Fcal_{s,\bsu}\\ \lVert f\rVert_{\Fcal_{s,\bsu}}\le 1}}\left| I(f)\right| ,
\end{align*}
which indeed equals 1 for any $s$ and $\bsu$. Hence the integration problem in $\Fcal_{s,\bsu}$ is well normalized. The worst-case error in $\Fcal_{s,\bsu}$ for a QMC rule using a point set $P$ is defined as
\begin{align*}
  e^{\wor}(\Fcal_{s,\bsu};P) = \sup_{\substack{f\in \Fcal_{s,\bsu}\\ \lVert f\rVert_{\Fcal_{s,\bsu}}\le 1}}\left| I(f;P)-I(f)\right| .
\end{align*}

We are interested in a dimension-independent super-polynomial convergence of the worst-case error of the form
\begin{equation}\label{eq:acc_conv}
e^{\wor}(\Fcal_{s,\bsu};P) \leq C e^{-c (\log{n})^{p}} \qquad \text{for all $n, s \in \NN$},
\end{equation}
where $C$ and $c$ are positive constants independent of $n$ and $s$.
The following existence result is from \cite{Suzuki15_2}.
%\begin{definition}\label{def:acc_conv}
%We say that a QMC rule using a point set $P$ with $|P|=N$ achieves an \emph{accelerating convergence} of the worst-case error in $\Fcal_{s,\bsu}$ if there exists a constant $w\in (0,1)$ and functions $C,C_1: \NN\to (0,\infty)$ and $p:\NN \to (1,\infty)$ such that 
%\begin{align}\label{eq:acc_conv}
%  e^{\wor}(\Fcal_{s,\bsu};P) \le C(s)w^{(\log N/C_1(s))^{p(s)}}.
%\end{align}
%Moreover, if the functions $C,C_1$ and $p$ are independent of $s$ in (\ref{eq:acc_conv}), a QMC rule is said to achieve a dimension-independent accelerating convergence of the worst-case error in $\Fcal_{s,\bsu}$.
%\end{definition}

%We can rewrite the right-hand side of (\ref{eq:acc_conv}) into
%\begin{align*}
% C(s)N^{-(\log w^{-1})(\log N)^{p(s)-1}/C_1(s)^{p(s)}}.
%\end{align*}
%Assuming $p$ is greater than 1, the worst-case error converges to zero with an increasing order with respect to $N$, and asymptotically converges faster than any polynomial in $N$. This is why we call this behavior accelerating convergence. The following existence result is from \cite{Suzuki15_2}.

\begin{theorem}[\cite{Suzuki15_2}]\label{thm:suzuki}
Consider the integration problem in the weighted function space $\Fcal_{s,\bsu}$ for a sequence of weights $\bsu$. If $\bsu$ satisfies $\liminf_{j\to \infty}\log(u_j^{-1})/j^r>0$ for $r>0$, then there exists a QMC rule which achieves a dimension-independent super-polynomial convergence of the worst-case error in $\Fcal_{s,\bsu}$ as \eqref{eq:acc_conv} with $p=(2r+1)/(r+1)$.
\end{theorem}

%%%%%%%%%%%%%%%%%%%%%%%%%%%%%%%%%%%%%%%%%%%%%%%%%%%%%%%%%%%
\subsection{Interlaced polynomial lattice rules}\label{subsec:ipls}

Here we give the definition of interlaced polynomial lattice rules, which are based on polynomial lattice rules, introduced by Niederreiter \cite{Nied92}, and a digit interlacing composition, introduced by Dick \cite{Dick07,Dick08}. 

We first introduce polynomial lattice rules. In this subsection, let $b$ be a prime number, and let $\ZZ_b$ be the finite field with $b$ elements. We denote by $\ZZ_b[x]$ the set of all polynomials over $\ZZ_b$, and denote by $\ZZ_b((x^{-1}))$ the field of formal Laurent series over $\ZZ_b$. Every element of $\ZZ_b((x^{-1}))$ can be represented as
\begin{align*}
  L = \sum_{l=w}^{\infty}t_l x^{-l} ,
\end{align*}
for some integer $w$ and $t_l\in \ZZ_b$ for all $l$. For a given integer $m$, we define the mapping $v_m$ from $\ZZ_b((x^{-1}))$ to the interval $[0,1)$ by
  \begin{align*}
    v_m\left( \sum_{l=w}^{\infty}t_l x^{-l} \right) =\sum_{l=\max(1,w)}^{m}t_l b^{-l}.
  \end{align*}
A non-negative integer $k$ whose $b$-adic expansion is given by $k=\kappa_0+\kappa_1 b+\cdots +\kappa_{a-1} b^{a-1}$ will be identified with the polynomial $k(x)=\kappa_0+\kappa_1 x+\cdots +\kappa_{a-1} x^{a-1}\in \ZZ_b[x]$.  For $\bsk=(k_1,\ldots, k_s)\in (\ZZ_b[x])^s$ and $\bsq=(q_1,\ldots, q_s)\in (\ZZ_b[x])^s$, we define the inner product as
  \begin{align}\label{eq:inner_product}
     \bsk \cdot \bsq := \sum_{j=1}^{s}k_j q_j \in \ZZ_b[x] ,
  \end{align}
and we write $q\equiv 0 \pmod p$ if $p$ divides $q$ in $\ZZ_b[x]$. Using this notation, polynomial lattice rules are constructed as follows.

\begin{definition}\label{def:polynomial_lattice}
Let $m, s \in \NN$. Let $p \in \ZZ_b[x]$ such that $\deg(p)=m$ and let $\bsq=(q_1,\ldots,q_s) \in (\ZZ_b[x])^s$. A polynomial lattice point set $P(\bsq,p)$ is a set consisting of $b^m$ points $\bsx_0,\ldots,\bsx_{b^m-1}$ that are defined as
  \begin{align*}
    \bsx_n := \left( v_m\left( \frac{n(x)q_1(x)}{p(x)} \right) , \ldots , v_m\left( \frac{n(x)q_s(x)}{p(x)} \right) \right) \in [0,1)^s ,
  \end{align*}
for $0\le n<b^m$. A QMC rule using this point set is called a \emph{polynomial lattice rule} with generating vector $\bsq$ and modulus $p$.
\end{definition}

We add one more notation and introduce the concept of the so-called \emph{dual polynomial lattice} of a polynomial lattice point set. For $k\in \NN_0$ with its $b$-adic expansion $k= \kappa_0 + \kappa_1 b+\cdots + \kappa_{a-1} b^{a-1}$, let $\rtr_m(k)$ be the polynomial of degree at most $m$ obtained by truncating the associated polynomial $k(x)\in \ZZ_b[x]$ as
  \begin{align*}
    \rtr_m(k)= \kappa_0 + \kappa_1 x+\cdots + \kappa_{m-1}x^{m-1},
  \end{align*}
where we set $\kappa_{a} = \cdots = \kappa_{m-1} = 0$ if $a< m$. For a vector $\bsk=(k_1,\ldots, k_s)\in \NN_0^s$, we define $\rtr_m(\bsk)=(\rtr_m(k_1),\ldots, \rtr_m(k_s))$. With this notation, we introduce the following definition of the dual polynomial lattice $P^{\perp}(\bsq,p)$.
\begin{definition}\label{def:dual_net}
The dual polynomial lattice of a polynomial lattice point set with modulus $p\in \ZZ_b[x]$, $\deg(p)=m$, and generating vector $\bsq \in (\ZZ_b[x])^s$ is given by
  \begin{align*}
     P^{\perp}(\bsq,p)  = \{ \bsk\in \NN_0^s:\ \mathrm{tr}_m(\bsk)\cdot \bsq\equiv 0 \pmod p \} ,
  \end{align*}
where the inner product is in the sense of (\ref{eq:inner_product}).
\end{definition}
\noindent
The following important lemma relates the dual polynomial lattice to numerical integration of Walsh functions, see \cite[Lemmas~4.75 and 10.6]{DPbook} for the proof.
\begin{lemma}\label{lem:dual_walsh}
Let $P(\bsq,p)=\{\bsx_0,\bsx_1,\dots,\bsx_{b^m-1}\}\subset [0,1)^{s}$ be a polynomial lattice point set with modulus $p\in \ZZ_b[x]$, $\deg(p)=m$, and generating vector $\bsq \in (\ZZ_b[x])^s$, and let $P^{\perp}(\bsq,p)$ be its dual polynomial lattice. Then we have
  \begin{align*}
    \frac{1}{b^m}\sum_{n=0}^{b^m-1}\wal_{\bsk}(\bsx_n)= \begin{cases}
     1 & \text{if $\bsk\in P^{\perp}(\bsq,p)$,} \\
     0 & \text{otherwise.} \\
     \end{cases}
  \end{align*}
\end{lemma}

We introduce the digit interlacing composition next. Let $d$ be a positive integer called interlacing factor, and let $\bsx=(x_1,\ldots,x_d)$ be a generic point in $[0,1)^d$ whose unique $b$-adic expansions are given by $x_j=\sum_{i=1}^{\infty}\xi_{i,j} b^{-i}$. Then the digit interlacing function $\Dcal_d: [0,1)^d \to [0,1)$ is defined as
  \begin{align*}
    \Dcal_d(\bsx) := \sum_{i=1}^{\infty}\sum_{j=1}^{d}\xi_{i,j}b^{-d(i-1)-j} .
  \end{align*}
We also define such a function for $ds$-dimensional vectors $\bsx=(x_1,\ldots,x_{ds})$ by applying $\Dcal_d$ to every consecutive $d$ components, that is,
  \begin{align*}
    \Dcal_d(\bsx) := \left( \Dcal_d(x_1,\dots,x_d), \Dcal_d(x_{d+1},\dots,x_{2d}), \ldots, \Dcal_d(x_{d(s-1)+1},\dots,x_{ds})\right) .
  \end{align*}
Now we are ready to introduce the definition of interlaced polynomial lattice rules \cite{Goda15,Goda15_2,GD15}.

\begin{definition}\label{def:interlacing_polynomial_lattice}
Let $m,s,d \in \NN$. Let $p \in \ZZ_b[x]$ such that $\deg(p)=m$ and let $\bsq=(q_1,\ldots,q_{ds}) \in (\ZZ_b[x])^{ds}$. An interlaced polynomial lattice point set $\Dcal_d(P(\bsq,p))$ of order $d$ is a set consisting of $b^m$ points defined as
  \begin{align*}
    \Dcal_d(P(\bsq,p)) := \left\{ \Dcal_d(\bsx) : \bsx \in P(\bsq,p) \right\} .
  \end{align*}
A QMC rule using this point set is called an \emph{interlaced polynomial lattice rule} of order $d$ with generating vector $\bsq$ and modulus $p$.
\end{definition}

%%%%%%%%%%%%%%%%%%%%%%%%%%%%%%%%%%%%%%%%%%%%%%%%%%%%%%%%%%%
\subsection{The results}\label{subsec:result}
We now describe the main results of this paper. In the following, let $b$ be a prime number and let $m,s,d\in \NN$. For $p \in \ZZ_b[x]$ with $\deg(p)=m$ and $\bsq=(q_1,\ldots,q_{ds}) \in (\ZZ_b[x])^{ds}$, we denote the polynomial lattice point set by $P(\bsq,p)=\{\bsx_0,\dots,\bsx_{b^m-1}\}\subset [0,1)^{ds}$ with $\bsx_n=(x_{n,1},x_{n,2},\dots,x_{n,ds})$, and denote the $b$-adic expansion of $x_{n,j}$ by $x_{n,j}=\sum_{i=1}^{\infty}\xi_{i,n,j} b^{-i}$ for $0\le n<b^m$ and $1\le j\le ds$. Moreover, we denote the interlaced polynomial lattice point set by $\Dcal_d(P(\bsq,p))=\{\bsy_0,\dots,\bsy_{b^m-1}\}\subset [0,1)^{s}$, where $\bsy_n=\Dcal_d(\bsx_n)$ for $0\le n<b^m$.

Let $\bsu$ be a sequence of weights, and as in Proposition~\ref{prop:walsh_decay}, let $\bsa = (a_j)_{j\in \NN}$ be the sequence given by \[a_j = -\log_b( C_b m_b^{-1} u_j), \quad j \in \mathbb{N}.\] In Section~\ref{sec:error}, we show that the worst-case error in $\Fcal_{s,\bsu}$ by a QMC rule using $\Dcal_d(P(\bsq,p))$ as quadrature points is bounded by
\begin{align*}
  e^{\wor}(\Fcal_{s,\bsu};\Dcal_d(P(\bsq,p))) \le C_{\bsu}-1+C_{\bsu}B_{\bsu}(\bsq,p) ,
\end{align*}
where $C_{\bsu}$ and $B_{\bsu}(\bsq,p)$ are given by
\begin{align*}
  C_{\bsu}=\prod_{j=1}^{s}\prod_{h=1}^{d}\prod_{i=m+1}^{\infty}\left\{ 1+\frac{b-1}{b^{d(i-1)+h+a_j}}\right\} ,
\end{align*}
and
\begin{align*}
  B_{\bsu}(\bsq,p) = -1+\frac{1}{b^m}\sum_{n=0}^{b^m-1}\prod_{j=1}^{s}\prod_{h=1}^{d}\prod_{i=1}^{m}\left\{ 1+\frac{\eta(\xi_{i,n,d(j-1)+h})}{b^{d(i-1)+h+a_j}}\right\} ,
\end{align*}
respectively, where $\eta: \{0,1,\dots,b-1\}\to \RR$ is defined as
\begin{align*}
  \eta(\xi) := \begin{cases}
  b-1 & \text{if $\xi=0$,} \\
  -1  & \text{otherwise.}
  \end{cases}
\end{align*}
Since $C_{\bsu}$ is independent of the modulus $p$ and generating vector $\bsq$, $B_{\bsu}(\bsq,p)$ can be used as a quality criterion for searching for good $p$ and $\bsq$. In the following we introduce the CBC algorithm.

We restrict $q_j$, $1\le j\le ds$, to non-zero polynomials over $\ZZ_b$ with its degree less than $m$, where $m=\deg(p)$. Provided that $p$ is irreducible, we can set $q_1=1$ without loss of generality. We denote by $R_{b,m}$ the set of all non-zero polynomials over $\ZZ_b$ with degree less than $m$, i.e.,
	\begin{align*}
		R_{b,m}=\{ q\in \ZZ_b[x]: \deg(q)<m\ \text{and}\ q\ne0\} .
	\end{align*}
We note that $|R_{b,m}|=b^m-1$. Further, we write $\bsq_{\tau}=(q_1,\ldots, q_{\tau})$ for $1\le \tau\le ds$. The idea is now to search for the polynomials $q_j \in R_{b,m}$ component-by-component. To do so, we need to define $B_{\bsu}(\bsq_{\tau},p)$ for arbitrary $1 \le \tau \le ds$. This is done in the following way. Let $1 \le \tau \le ds$ and $\beta = \lceil \tau/d \rceil$. Then
\begin{align}\label{eq:B_u_tau}
  B_{\bsu}(\bsq_{\tau},p) & = -1+\frac{1}{b^m}\sum_{n=0}^{b^m-1}\prod_{j=1}^{\beta-1}\prod_{h=1}^{d}\prod_{i=1}^{m}\left\{ 1+\frac{\eta(\xi_{i,n,d(j-1)+h})}{b^{d(i-1)+h+a_j}}\right\} \nonumber \\
  & \qquad \times \prod_{h=1}^{\tau -d(\beta-1)}\prod_{i=1}^{m}\left\{ 1+\frac{\eta(\xi_{i,n,d(\beta-1)+h})}{b^{d(i-1)+h+a_{\beta}}}\right\}.
\end{align}

The CBC construction proceeds as follows.
\begin{algorithm}\label{algorithm:cbc}
Let $b,m,s,d,\bsu=(u_j)_{j\in \NN}$ be as above.
	\begin{enumerate}
		\item Choose an irreducible polynomial $p\in \ZZ_b[x]$ with $\deg(p)=m$.
		\item Set $q_1=1$.
		\item For $\tau=2,\ldots, ds$, find $q_{\tau}$ by minimizing $B_{\bsu}((\bsq_{\tau-1},q),p)$ as a function of $q\in R_{b,m}$.
	\end{enumerate}
\end{algorithm}
\noindent
In Subsection~\ref{subsec:fast_cbc}, we show that one can also use the fast CBC algorithm of \cite{NC06,NC06_2} to find good generating vectors.

Next we show that the generating vector $\bsq$ found by Algorithm~\ref{algorithm:cbc} satisfies the following bound.

\begin{theorem}\label{thm:cbc_bound}
Let $b$ be a prime and $p\in \ZZ_b[x]$ be irreducible with $\deg(p)=m$. Let $\phi:[0,\infty)\to [0,\infty)$ be a concave and unbounded monotonic increasing function. Suppose that $\bsq=(q_1,\ldots, q_{ds})$ is constructed using Algorithm \ref{algorithm:cbc}. Then we have
\begin{align*}
  B_{\bsu}(\bsq,p) \le \phi^{-1}\left[ \frac{1}{b^m-1}\sum_{\substack{\bsk\in \NN_0^s\setminus \{\bszero\}\\ k_j<b^{dm}, \forall j}}\phi\left(b^{-\mu(\bsa;\bsk)}\right)\right] .
\end{align*}
\end{theorem}
\noindent 
The proof of this result is presented in Subsection~\ref{subsec:cbc}. 

The function $\phi(x)=x^{\lambda}$, $0<\lambda\le 1$, has been often used to obtain these types of error bounds in the literature. In this case, one may apply so-called Jensen's inequality
\begin{align}\label{eq:jensen}
  \phi\left( \sum_{n}c_n\right) \le \sum_{n}\phi(c_n),
\end{align}
for any sequence of non-negative real numbers $(c_n)$. The inequality (\ref{eq:jensen}), however, also holds for any concave function $\phi:[0,\infty)\to [0,\infty)$ \cite[Section~2.3]{DHP15}. In our case, the function $\phi(x)=x^{\lambda}$ is not a good choice because it does not give us the worst-case error bound with a super-polynomial convergence. Instead we use $\phi$ which maps $b^{-\mu(\bsa;\bsk)}$ to $b^{-(\mu(\bsa;\bsk))^{\lambda}}$ for $0<\lambda\le 1$. Such a map can be designed as follows. For $b=2$, let $\tilde{x}_{\lambda}=2^{-(\log 2)^{1/\lambda}}$ for $0<\lambda\le 1$. Then
\begin{align}\label{eq:jensen_b_2}
  \phi(x) = \begin{cases}
  2^{-(\log_2 (1/x))^{1/\lambda}} & \text{if $0< x< \tilde{x}_{\lambda}$,} \\
  \frac{\lambda \log_2(\tilde{x}_{\lambda})^{\lambda-1}}{e\tilde{x}_{\lambda}}(x-\tilde{x}_{\lambda})+\frac{1}{e} & \text{if $x\ge  \tilde{x}_{\lambda}$.}
  \end{cases}
\end{align}
For $b\ge 3$, 
\begin{align}\label{eq:jensen_b_geq_3}
  \phi(x) = \begin{cases}
  b^{-(\log_b (1/x))^{1/\lambda}} & \text{if $0< x< \frac{1}{b}$,} \\
  \lambda \left( x-\frac{1}{b}\right)+\frac{1}{b} & \text{if $x\ge  \frac{1}{b}$.}
  \end{cases}
\end{align}
Note that we set $\phi(0)=0$ for any $b$ and $0<\lambda\le 1$, and that the function $\phi$ is concave and unbounded monotonic increasing on $[0,\infty)$. As above we need a slight modification for the case $b=2$ since the function $\phi(x) =2^{-(\log_2 (1/x))^{1/\lambda}}$ is concave over the interval $(0,\tilde{x}_{\lambda})$ but not over the interval $(0,1/2)$. Using this function and under the same condition on the weights with Theorem~\ref{thm:suzuki}, we have the following corollary of Theorem~\ref{thm:cbc_bound}.

\begin{corollary}\label{cor:cbc_bound}
Assume that $\bsu$ satisfies $\liminf_{j\to \infty}\log(u_j^{-1})/j^r >0$ for $r>0$. Let $b$ be a prime and $p\in \ZZ_b[x]$ be irreducible with $\deg(p)=m$. Suppose that $\bsq=(q_1,\ldots, q_{ds})$ is constructed using Algorithm \ref{algorithm:cbc}. Then there exist constants $D_{r,\lambda},E_{r,\lambda}>0$ both independent of $s$ such that we have
\begin{align*}
  B_{\bsu}(\bsq,p) \le E_{r,\lambda}b^{-\left(\log_b\left(\frac{b^m-1}{D_{r,\lambda}}\right)\right)^{1/\lambda}} ,
\end{align*}
for any $(r+1)/(2r+1)<\lambda \le 1$. Moreover, by setting $d\ge m^{r/(r+1)}$, the worst-case error satisfies the bound
\begin{align*}
  e^{\wor}(\Fcal_{s,\bsu};\Dcal_d(P(\bsq,p)))  \le E'_{r,\lambda}b^{-\left(\log_b\left(\frac{b^m-1}{D_{r,\lambda}+1}\right)\right)^{1/\lambda}} ,
\end{align*}
where $E'_{r,\lambda}>0$ is a constant independent of $s$.
\end{corollary}
\noindent 
The proof of this result is presented in Subsection~\ref{subsec:dependence}. 

This result means that we can construct a QMC rule which achieves a dimension-independent super-polynomial convergence of the worst-case error in $\Fcal_{s,\bsu}$ as \eqref{eq:acc_conv} with $1<p<(2r+1)/(r+1)$. This is a bit weaker than Theorem~\ref{thm:suzuki} (shown by Suzuki in \cite{Suzuki15_2}), since we do not have an error bound for the endpoint $p=(2r+1)/(r+1)$. Under an additional assumption, however, it is even possible to include the case $\lambda=(r+1)/(2r+1)$ in Corollary~\ref{cor:cbc_bound}, see Remark~\ref{rm:endpoint}. The most important advantage of our approach is that a good QMC rule can be explicitly constructed by using a CBC algorithm.

%%%%%%%%%%%%%%%%%%%%%%%%%%%%%%%%%%%%%%%%%%%%%%%%%%%%%%%%%%%
%%%%%%%%%%%%%%%%%%%%%%%%%%%%%%%%%%%%%%%%%%%%%%%%%%%%%%%%%%%
%%%%%%%%%%%%%%%%%%%%%%%%%%%%%%%%%%%%%%%%%%%%%%%%%%%%%%%%%%%
\section{The worst-case error in $\Fcal_{s,\bsu}$}\label{sec:error}
To analyze the worst-case error of interlaced polynomial lattice rules, we introduce a digit interlacing composition for non-negative integers. Let $d$ be an interlacing factor, and let $\bsk=(k_1,\ldots,k_d)\in \NN_0^d$ whose $b$-adic expansions are given by $k_j=\sum_{i=0}^{\infty}\kappa_{i,j} b^i$, which is actually a finite expansion. Then the digit interlacing function $\Ecal_d: \NN_0^d \to \NN_0$ is defined as
  \begin{align*}
    \Ecal_d(\bsk) := \sum_{i=0}^{\infty}\sum_{j=1}^{d}\kappa_{i,j}b^{di+j-1} .
  \end{align*}
It is obvious to show that $\Ecal_d$ is bijective. We also define such a function for $ds$-dimensional vectors $\bsk=(k_1,\ldots,k_{ds})\in \NN_0^{ds}$ by applying $\Ecal_d$ to every consecutive $d$ components, that is,
  \begin{align*}
    \Ecal_d(\bsk) := \left( \Ecal_d(k_1,\dots,k_d), \Ecal_d(k_{d+1},\dots,k_{2d}), \ldots, \Ecal_d(k_{d(s-1)+1},\dots,k_{ds})\right) .
  \end{align*}
The following lemma relates an interlaced polynomial lattice point set to numerical integration of Walsh functions, see \cite[Lemma~1]{Goda15} for the proof.

\begin{lemma}\label{lem:interlaced_Walsh}
Let $\Dcal_d(P(\bsq,p))=\{\bsy_0,\bsy_1,\dots,\bsy_{b^m-1}\}\subset [0,1)^s$ be an interlaced polynomial lattice point set of order $d$ with modulus $p\in \ZZ_b[x]$, $\deg(p)=m$, and generating vector $\bsq \in (\ZZ_b[x])^{ds}$. For $\bsk\in \NN_0^{ds}$, we have
  \begin{align*}
    \frac{1}{b^m}\sum_{n=0}^{b^m-1}\wal_{\Ecal_d(\bsk)}(\bsy_n)= \begin{cases}
     1 & \text{if $\bsk\in P^{\perp}(\bsq,p)$,} \\
     0 & \text{otherwise.} \\
     \end{cases}
  \end{align*}
\end{lemma}

We introduce another function $\tilde{\mu}(a,h;\cdot): \NN_0\to \RR$ for a real number $a$ and an integer $h\in \{1,2,\dots, d\}$. For $k\in \NN$, we denote its $b$-adic expansion by $k = \kappa_1 b^{c_1-1}+\kappa_2 b^{c_2-1}+\cdots +\kappa_v b^{c_v-1}$ such that $\kappa_1,\dots,\kappa_v \in \{1,2,\dots,b-1\}$ and $c_1>\ldots > c_v >0$. The function $\tilde{\mu}(a,h;\cdot): \NN_0\to \RR$ is defined as 
  \begin{align*}
    \tilde{\mu}(a,h;k) := \sum_{i=1}^{v}\left[ d(c_i-1)+h+a\right] ,
  \end{align*}
and $\tilde{\mu}(a,h;0):=0$. For vectors of real numbers $\bsa$ and $\bsk\in \NN_0^{ds}$, we define
  \begin{align*}
    \tilde{\mu}(\bsa;\bsk) := \sum_{j=1}^{s}\sum_{h=1}^{d}\tilde{\mu}(a_j,h; k_{d(j-1)+h}) .
  \end{align*}
With a slight abuse of notation, for $u\subset \{1,2,\ldots,ds\}$ and $\bsk_u\in \NN_0^s$, we write $\tilde{\mu}(\bsa;\bsk_u):= \tilde{\mu}(\bsa;(\bsk_u,\bszero))$, where the vector $(\bsk_u,\bszero)$ denotes the $ds$-dimensional vector whose $j$-th component is $k_j$ for $j\in u$ and $0$ otherwise. From Definition~\ref{def:weight_s} and the definition of $\Ecal_d$, we have 
\begin{align}\label{eq:metric_identity}
   \mu(\bsa;\Ecal_d(\bsk)) = \tilde{\mu}(\bsa;\bsk) .
\end{align}

Now the worst-case error for numerical integration in $\Fcal_{s,\bsu}$ using an interlaced polynomial lattice rule is given as follows.

\begin{proposition}\label{prop:worst-case_error_interlaced}
Let $\Dcal_d(P(\bsq,p))\subset [0,1)^s$ be an interlaced polynomial lattice point set of order $d$ with modulus $p\in \ZZ_b[x]$, $\deg(p)=m$, and generating vector $\bsq \in (\ZZ_b[x])^{ds}$. For a sequence of the weights $\bsu$, we have
\begin{align*}
  e^{\wor}(\Fcal_{s,\bsu};\Dcal_d(P(\bsq,p))) \le \sum_{\bsk\in P^{\perp}(\bsq,p)\setminus \{\bszero\}}b^{-\tilde{\mu}(\bsa;\bsk)} ,
\end{align*}
where $P^{\perp}(\bsq,p)$ is the dual polynomial lattice of $P(\bsq,p)$, and $\bsa$ is a sequence of real numbers given as in Proposition~\ref{prop:walsh_decay}.
\end{proposition}

\begin{proof}
We write $\Dcal_d(P(\bsq,p))=\{\bsy_0,\bsy_1,\dots,\bsy_{b^m-1}\}$. Let us consider a function $f\in \Fcal_{s,\bsu}$. Given the Walsh series expansion of $f$ and the fact that $\Ecal_d$ is bijective, the signed integration error becomes
\begin{align*}
I(f;\Dcal_d(P(\bsq,p)))-I(f) & = \frac{1}{b^m}\sum_{n=0}^{b^m-1}f(\bsy_n)- \hat{f}(\bszero) \\
& = \frac{1}{b^m}\sum_{n=0}^{b^m-1}\sum_{\bsk\in \NN_0^{ds}}\hat{f}(\Ecal_d(\bsk))\wal_{\Ecal_d(\bsk)}(\bsy_n) - \hat{f}(\bszero) \\
& = \sum_{\bsk\in \NN_0^{ds}}\hat{f}(\Ecal_d(\bsk))\frac{1}{b^m}\sum_{n=0}^{b^m-1}\wal_{\Ecal_d(\bsk)}(\bsy_n) - \hat{f}(\bszero) \\
& = \sum_{\bsk\in P^{\perp}(\bsq,p)\setminus \{\bszero\}}\hat{f}(\Ecal_d(\bsk)) ,
\end{align*}
where we use Lemma~\ref{lem:interlaced_Walsh} in the last equality. Then we obtain
\begin{align*}
e^{\wor}(\Fcal_{s,\bsu};\Dcal_d(P(\bsq,p))) & = \sup_{\substack{f\in \Fcal_{s,\bsu}\\ \lVert f\rVert_{\Fcal_{s,\bsu}} \le 1}}\left| \sum_{\bsk\in P^{\perp}(\bsq,p)\setminus \{\bszero\}}\hat{f}(\Ecal_d(\bsk))\right| \nonumber \\
& \le \sup_{\substack{f\in \Fcal_{s,\bsu}\\ \lVert f\rVert_{\Fcal_{s,\bsu}} \le 1}}\sum_{\bsk\in P^{\perp}(\bsq,p)\setminus \{\bszero\}}\left|\hat{f}(\Ecal_d(\bsk))\right| \nonumber \\
& \le \sup_{\substack{f\in \Fcal_{s,\bsu}\\ \lVert f\rVert_{\Fcal_{s,\bsu}} \le 1}}\lVert f\rVert_{\Fcal_{s,\bsu}}\sum_{\bsk\in P^{\perp}(\bsq,p)\setminus \{\bszero\}}b^{-\mu(\bsa;\Ecal_d(\bsk))} \nonumber \\
& = \sum_{\bsk\in P^{\perp}(\bsq,p)\setminus \{\bszero\}}b^{-\tilde{\mu}(\bsa;\bsk)} ,
\end{align*}
where we use the triangle inequality, Proposition~\ref{prop:walsh_decay} and the identity (\ref{eq:metric_identity}) in the first inequality, the second inequality and the last equality, respectively.
\end{proof}

Since the error bound in Proposition~\ref{prop:worst-case_error_interlaced} is independent of a particular function $f$, it can be used as a quality criterion for the construction of interlaced polynomial lattice rules. The following proposition gives a concise formula for the bound  on $e^{\wor}(\Fcal_{s,\bsu};\Dcal_d(P(\bsq,p)))$ in Proposition~\ref{prop:worst-case_error_interlaced}.

\begin{proposition}\label{prop:worst-case_error_interlaced2}
Let $P(\bsq,p)=\{\bsx_0,\bsx_1,\dots,\bsx_{b^m-1}\}\subset [0,1)^{ds}$ be a polynomial lattice point set with modulus $p\in \ZZ_b[x]$, $\deg(p)=m$, and generating vector $\bsq \in (\ZZ_b[x])^{ds}$. Let $\Dcal_d(P(\bsq,p))$ be its interlaced polynomial lattice point set. We denote the $b$-adic expansion of $x_{n,j}$ by $x_{n,j}=\sum_{i=1}^{\infty}\xi_{i,n,j} b^{-i}$ for $0\le n<b^m$ and $1\le j\le ds$. For a sequence of real numbers $\bsa$, we have
\begin{align*}
  \sum_{\bsk\in P^{\perp}(\bsq,p)\setminus \{\bszero\}}b^{-\tilde{\mu}(\bsa;\bsk)} = -1+\frac{1}{b^m}\sum_{n=0}^{b^m-1}\prod_{j=1}^{s}\prod_{h=1}^{d}\prod_{i=1}^{\infty}\left\{ 1+\frac{\eta(\xi_{i,n,d(j-1)+h})}{b^{d(i-1)+h+a_j}}\right\} ,
\end{align*}
where $\eta: \{0,1,\dots,b-1\}\to \RR$ is defined as
\begin{align*}
  \eta(\xi) := \begin{cases}
  b-1 & \text{if $\xi=0$,} \\
  -1  & \text{otherwise.}
  \end{cases}
\end{align*}
\end{proposition}

\begin{proof}
Using Lemma~\ref{lem:dual_walsh}, we have
\begin{align*}
& \sum_{\bsk\in P^{\perp}(\bsq,p)\setminus \{\bszero\}}b^{-\tilde{\mu}(\bsa;\bsk)} \\
& = \sum_{\bsk\in \NN_0^{ds}\setminus \{\bszero\}}b^{-\tilde{\mu}(\bsa;\bsk)}\frac{1}{b^m}\sum_{n=0}^{b^m-1}\wal_{\bsk}(\bsx_n) \\
& = -1+\frac{1}{b^m}\sum_{n=0}^{b^m-1}\sum_{\bsk\in \NN_0^{ds}}b^{-\tilde{\mu}(\bsa;\bsk)}\wal_{\bsk}(\bsx_n) \\
& = -1+\frac{1}{b^m}\sum_{n=0}^{b^m-1}\prod_{j=1}^{s}\prod_{h=1}^{d}\sum_{k_{d(j-1)+h}=0}^{\infty}b^{-\tilde{\mu}(a_j,h;k_{d(j-1)+h})}\wal_{k_{d(j-1)+h}}(x_{n,d(j-1)+h}) .
\end{align*}

Let us consider the function $\psi_{a,h,M}:[0,1)\to \RR$, for a real number $a$ and $h,M\in \NN$ with $1\le h\le d$, given by
\begin{align*}
\psi_{a,h,M}(x) = \sum_{k=0}^{b^M-1}b^{-\tilde{\mu}(a,h;k)}\wal_{k}(x) ,
\end{align*}
for $x\in [0,1)$. Denoting the unique $b$-adic expansion of $x\in [0,1)$ by $x=\sum_{i=1}^{\infty}\xi_{i} b^{-i}$, and denoting also the $b$-adic expansion of $k\in \NN_0$, $0\le k<b^M$, by $k=\sum_{i=0}^{M-1}\kappa_{i} b^{i}$, we have
\begin{align*}
\tilde{\mu}(a,h;\kappa_0+\kappa_1 b+\dots+\kappa_{M-1}b^{M-1}) = \sum_{i=1}^{M}\left[ d(i-1)+h+a\right]\chi(\kappa_{i-1}\ne 0) ,
\end{align*}
where $\chi$ denotes the indicator function, and
\begin{align*}
\wal_{k}(x) = \prod_{i=1}^{M}\omega_b^{\kappa_{i-1} \xi_i} .
\end{align*}
Then we have
\begin{align*}
\psi_{a,h,M}(x) & = \sum_{\kappa_0=0}^{b-1}\sum_{\kappa_1=0}^{b-1}\cdots \sum_{\kappa_{M-1}=0}^{b-1}\prod_{i=1}^{M}b^{-\left[ d(i-1)+h+a\right]\chi(\kappa_{i-1}\ne 0)}\omega_b^{\kappa_{i-1} \xi_i} \\
& = \prod_{i=1}^{M}\sum_{\kappa_{i-1}=0}^{b-1}b^{-\left[ d(i-1)+h+a\right]\chi(\kappa_{i-1}\ne 0)}\omega_b^{\kappa_{i-1} \xi_i} \\
& = \prod_{i=1}^{M}\left\{ 1+\frac{\eta(\xi_i)}{b^{d(i-1)+h+a}}\right\} .
\end{align*}
%Now it is straightforward to have the following pointwise absolute convergence of $\psi_{a,h,M}$ as $M$ goes to infinity,
By letting $M$ go to $\infty$ we obtain that
\begin{align*}
\sum_{k=0}^{\infty}b^{-\tilde{\mu}(a,h;k)}\wal_{k}(x) = \lim_{M\to \infty}\psi_{a,h,M}(x) = \prod_{i=1}^{\infty}\left\{ 1+\frac{\eta(\xi_i)}{b^{d(i-1)+h+a}}\right\} ,
\end{align*}
which converges pointwise absolutely. Hence the result follows.
\end{proof}

In order to evaluate the bound on $e^{\wor}(\Fcal_{s,\bsu};\Dcal_d(P(\bsq,p)))$ as shown in Proposition~\ref{prop:worst-case_error_interlaced2}, one needs to compute an infinite product. This infeasible computation can be avoided in the following way. As already stated in Subsection~\ref{subsec:result}, we have $\xi_{m+1,n,j}=\xi_{m+2,n,j}=\cdots =0$ for any $0\le n<b^m$ and $1\le j\le ds$. By setting
\begin{align*}
  C_{\bsu}=\prod_{j=1}^{s}\prod_{h=1}^{d}\prod_{i=m+1}^{\infty}\left\{ 1+\frac{b-1}{b^{d(i-1)+h+a_j}}\right\} ,
\end{align*}
we have
\begin{align}\label{bound_e_B}
  & \quad e^{\wor}(\Fcal_{s,\bsu};\Dcal_d(P(\bsq,p))) \nonumber \\
  & \le -1+\frac{C_{\bsu}}{b^m}\sum_{n=0}^{b^m-1}\prod_{j=1}^{s}\prod_{h=1}^{d}\prod_{i=1}^{m}\left\{ 1+\frac{\eta(\xi_{i,n,d(j-1)+h})}{b^{d(i-1)+h+a_j}}\right\} \nonumber \\
  & = C_{\bsu}-1+C_{\bsu}\left[ -1+\frac{1}{b^m}\sum_{n=0}^{b^m-1}\prod_{j=1}^{s}\prod_{h=1}^{d}\prod_{i=1}^{m}\left\{ 1+\frac{\eta(\xi_{i,n,d(j-1)+h})}{b^{d(i-1)+h+a_j}}\right\}\right] \nonumber \\
  & =: C_{\bsu}-1+C_{\bsu}B_{\bsu}(\bsq,p).
\end{align}
In \eqref{bound_Cu} below we show that $C_{\bsu}-1 \le E b^{-dm}$ for some constant $E > 0$, so that the main term in \eqref{bound_e_B} is $B_{\bsu}(\bsq, p)$. Since an infinite product does not appear in $B_{\bsu}(\bsq,p)$, we can use $B_{\bsu}(\bsq,p)$ as a quality criterion for searching for good generating vectors $\bsq$ instead of $e^{\wor}(\Fcal_{s,\bsu};\Dcal_d(P(\bsq,p)))$. Note that $B_{\bsu}(\bsq,p)$ can be also expressed as
\begin{align*}
  B_{\bsu}(\bsq,p) = \sum_{\substack{\bsk\in P^{\perp}(\bsq,p)\setminus \{\bszero\}\\ k_j<b^m, \forall j}}b^{-\tilde{\mu}(\bsa;\bsk)} .
\end{align*}

%%%%%%%%%%%%%%%%%%%%%%%%%%%%%%%%%%%%%%%%%%%%%%%%%%%%%%%%%%%
%%%%%%%%%%%%%%%%%%%%%%%%%%%%%%%%%%%%%%%%%%%%%%%%%%%%%%%%%%%
%%%%%%%%%%%%%%%%%%%%%%%%%%%%%%%%%%%%%%%%%%%%%%%%%%%%%%%%%%%
\section{Component-by-component algorithm}\label{sec:cbc}
\subsection{Error bounds for the algorithm}\label{subsec:cbc}
In the proof of Theorem~\ref{thm:cbc_bound} and its subsequent analysis, we use Jensen's inequality \eqref{eq:jensen} for a concave and unbounded monotonic increasing function $\phi:[0,\infty)\to [0,\infty)$.

As mentioned in Subsection~\ref{subsec:result}, for arbitrary $1\le \tau\le ds$ , we define $B_{\bsu}(\bsq_{\tau},p)$ as in (\ref{eq:B_u_tau}), which can be also expressed as
\begin{align*}
  B_{\bsu}(\bsq_{\tau},p) = \sum_{\substack{\bsk\in P^{\perp}(\bsq_{\tau},p)\setminus \{\bszero\}\\ k_j<b^m, \forall j}}b^{-\tilde{\mu}(\bsa;\bsk)} .
\end{align*}
In order to prove Theorem~\ref{thm:cbc_bound}, we first show the following proposition.
\begin{proposition}\label{prop:cbc_bound}
Let $b$ be a prime and $p\in \ZZ_b[x]$ be irreducible with $\deg(p)=m$. Let $\phi:[0,\infty)\to [0,\infty)$ be a concave and unbounded monotonic increasing function. Suppose that $\bsq=(q_1,\ldots, q_{ds})$ is constructed using Algorithm \ref{algorithm:cbc}. Then, for any $\tau=1,\dots,ds$, we have
\begin{align*}
  B_{\bsu}(\bsq_{\tau},p) \le \phi^{-1}\left[ \frac{1}{b^m-1}\sum_{\substack{\bsk\in \NN_0^\tau\setminus \{\bszero\}\\ k_j<b^{m}, \forall j}}\phi\left(b^{-\tilde{\mu}(\bsa;\bsk)}\right)\right] .
\end{align*}
\end{proposition}

\begin{proof}
We prove the proposition by induction on $\tau$. Let us consider the case $\tau=1$ first. Since $p\in \ZZ_b[x]$ is irreducible with $\deg(p)=m$ and $q_1=1$, we have $P^{\perp}(1,p)=\{b^m k: k\in \NN_0\}$, which implies
\begin{align*}
B_{\bsu}(1,p) = \sum_{\substack{k \in P^{\perp}(1,p)\setminus \{0\}\\ k <b^m}}b^{-\tilde{\mu}(a_1,1;k)}=0.
\end{align*}
Thus we have
\begin{align*}
B_{\bsu}(1,p) \le \phi^{-1}\left[ \frac{1}{b^m-1}\sum_{0< k <b^m}\phi\left( b^{-\tilde{\mu}(a_1,1;k)}\right)\right].
\end{align*}

Suppose that for some integer $1\le \tau<ds$ we have already obtained $\bsq_{\tau}\in (R_{b,m})^{\tau}$ such that
\begin{align*}
  B_{\bsu}(\bsq_{\tau},p) \le \phi^{-1}\left[ \frac{1}{b^m-1}\sum_{\substack{\bsk\in \NN_0^\tau\setminus \{\bszero\}\\ k_j<b^{m}, \forall j}}\phi\left(b^{-\tilde{\mu}(\bsa;\bsk)}\right)\right] .
\end{align*}
Now we consider a $q\in R_{b,m}$. We have
\begin{align*}
  B_{\bsu}((\bsq_{\tau},q),p) & = \sum_{\substack{\bsk\in P^{\perp}((\bsq_{\tau},q),p)\setminus \{\bszero\}\\ k_j<b^m, \forall j}}b^{-\tilde{\mu}(\bsa;\bsk)} \\
  & = \sum_{\substack{\bsk\in P^{\perp}((\bsq_{\tau},q),p)\setminus \{\bszero\}\\ \text{$k_j<b^m$ for $1\le j\le \tau$}\\ k_{\tau+1}=0}}b^{-\tilde{\mu}(\bsa;\bsk)} + \sum_{\substack{\bsk\in P^{\perp}((\bsq_{\tau},q),p)\setminus \{\bszero\}\\ \text{$k_j<b^m$ for $1\le j\le \tau$}\\ 0<k_{\tau+1}<b^m}}b^{-\tilde{\mu}(\bsa;\bsk)} \\
  & = B_{\bsu}(\bsq_{\tau},p)+\theta_{\bsu}((\bsq_{\tau},q),p) ,
\end{align*}
where we write
\begin{align*}
\theta_{\bsu}((\bsq_{\tau},q),p) := \sum_{\substack{\bsk\in P^{\perp}((\bsq_{\tau},q),p)\setminus \{\bszero\}\\ \text{$k_j<b^m$ for $1\le j\le \tau$}\\ 0<k_{\tau+1}<b^m}}b^{-\tilde{\mu}(\bsa;\bsk)}.
\end{align*}
To find $q_{\tau+1}\in R_{b,m}$ which minimizes $B_{\bsu}((\bsq_{\tau},q),p)$ as a function of $q$, we only need to consider the term $\theta_{\bsu}((\bsq_{\tau},q),p)$. Using an averaging argument and Jensen's inequality \eqref{eq:jensen} for a concave and unbounded monotonic increasing function $\phi:[0,\infty)\to [0,\infty)$, we have
\begin{align*}
\phi\left( \theta_{\bsu}(\bsq_{\tau+1},p)\right) & \le \frac{1}{b^m-1}\sum_{q\in R_{b,m}} \phi\left( \theta_{\bsu}((\bsq_{\tau},q),p)\right) \\
& \le \frac{1}{b^m-1}\sum_{q\in R_{b,m}} \sum_{\substack{\bsk\in P^{\perp}((\bsq_{\tau},q),p)\setminus \{\bszero\}\\ \text{$k_j<b^m$ for $1\le j\le \tau$}\\ 0<k_{\tau+1}<b^m}} \phi\left( b^{-\tilde{\mu}(\bsa;\bsk)} \right)\\
& = \sum_{\substack{\bsk\in \NN_0^{\tau+1}\setminus \{\bszero\}\\ \text{$k_j<b^m$ for $1\le j\le \tau$}\\ 0<k_{\tau+1}<b^m}}\frac{\phi\left( b^{-\tilde{\mu}(\bsa;\bsk)} \right)}{b^m-1}\sum_{\substack{q\in R_{b,m}\\ \rtr_m(\bsk)\cdot (\bsq_{\tau},q)\equiv 0 \pmod p}}1 .
\end{align*}
Since $k_{\tau+1}$ cannot be a multiple of $b^m$, the inner sum in the last expression equals $0$ if $(k_1,\dots,k_{\tau})\in P^{\perp}(\bsq_{\tau},p)$, and equals $1$ otherwise. Through this argument, we obtain
\begin{align*}
\phi\left( \theta_{\bsu}(\bsq_{\tau+1},p)\right) & \le \frac{1}{b^m-1}\sum_{\substack{\bsk\in \NN_0^{\tau+1}\setminus \{\bszero\}\\ \bsk_{\tau}\notin P^{\perp}(\bsq_{\tau},p) \\ \text{$k_j<b^m$ for $1\le j\le \tau$}\\ 0<k_{\tau+1}<b^m}}\phi\left( b^{-\tilde{\mu}(\bsa;\bsk)} \right) \\
& \le \frac{1}{b^m-1}\sum_{\substack{\bsk\in \NN_0^{\tau+1}\setminus \{\bszero\}\\ \text{$k_j<b^m$ for $1\le j\le \tau$}\\ 0<k_{\tau+1}<b^m}}\phi\left( b^{-\tilde{\mu}(\bsa;\bsk)} \right).
\end{align*}

Finally, using Jensen's inequality (\ref{eq:jensen}) again, we have
\begin{align*}
& \quad \phi\left( B_{\bsu}(\bsq_{\tau+1},p)\right) \\
& \le \phi\left( B_{\bsu}(\bsq_{\tau},p)\right) + \phi\left( \theta_{\bsu}(\bsq_{\tau+1},p) \right) \\
& \le \frac{1}{b^m-1}\sum_{\substack{\bsk\in \NN_0^\tau\setminus \{\bszero\}\\ k_j<b^{m}, \forall j}}\phi\left(b^{-\tilde{\mu}(\bsa;\bsk)}\right)+ \frac{1}{b^m-1} \sum_{\substack{\bsk\in \NN_0^{\tau+1}\setminus \{\bszero\}\\ \text{$k_j<b^m$ for $1\le j\le \tau$}\\ 0<k_{\tau+1}<b^m}}\phi\left( b^{-\tilde{\mu}(\bsa;\bsk)} \right) \\
& = \frac{1}{b^m-1}\sum_{\substack{\bsk\in \NN_0^{\tau+1}\setminus \{\bszero\}\\ k_j<b^{m}, \forall j}}\phi\left(b^{-\tilde{\mu}(\bsa;\bsk)}\right) .
\end{align*}
Hence the result follows.
\end{proof}

Now we are ready to prove Theorem~\ref{thm:cbc_bound}.
\begin{proof}[Proof of Theorem~\ref{thm:cbc_bound}]
From Proposition~\ref{prop:cbc_bound} in which let $\tau=ds$, and using the identity (\ref{eq:metric_identity}) and the fact that $\Ecal_d$ is bijective between $\{0,1,\ldots,b^m-1\}^{d}$ and $\{0,1,\ldots,b^{dm}-1\}$, we have
\begin{align*}
  B_{\bsu}(\bsq,p) & \le \phi^{-1}\left[ \frac{1}{b^m-1}\sum_{\substack{\bsk\in \NN_0^{ds}\setminus \{\bszero\}\\ k_j<b^{m}, \forall j}}\phi\left(b^{-\tilde{\mu}(\bsa;\bsk)}\right)\right] \\
  & = \phi^{-1}\left[ \frac{1}{b^m-1}\sum_{\substack{\bsk\in \NN_0^{ds}\setminus \{\bszero\}\\ k_j<b^{m}, \forall j}}\phi\left(b^{-\mu(\bsa;\Ecal_d(\bsk))}\right)\right] \\
  & = \phi^{-1}\left[ \frac{1}{b^m-1}\sum_{\substack{\bsk\in \NN_0^{s}\setminus \{\bszero\}\\ k_j<b^{dm}, \forall j}}\phi\left(b^{-\mu(\bsa;\bsk)}\right)\right] ,
\end{align*}
which implies the result.
\end{proof}

So far, we have proved the following bound on the worst-case error:
\begin{align}\label{eq:cbc_error_bound}
 & \quad e^{\wor}(\Fcal_{s,\bsu};\Dcal_d(P(\bsq,p))) \nonumber \\
 & \le C_{\bsu}-1+C_{\bsu}\phi^{-1}\left[ \frac{1}{b^m-1}\sum_{\substack{\bsk\in \NN_0^{s}\setminus \{\bszero\}\\ k_j<b^{dm}, \forall j}}\phi\left(b^{-\mu(\bsa;\bsk)}\right)\right] ,
\end{align}
for a concave and unbounded monotonic increasing function $\phi:[0,\infty)\to [0,\infty)$. In the following section we investigate the error bound in more detail using the function $\phi$ defined in \eqref{eq:jensen_b_2} for $b=2$ and in \eqref{eq:jensen_b_geq_3} for $b\ge 3$.

%%%%%%%%%%%%%%%%%%%%%%%%%%%%%%%%%%%%%%%%%%%%%%%%%%%%%%%%%%%
\subsection{Dependence of the error bounds on the dimension}\label{subsec:dependence}
Here we study the dependence of the worst-case error bounds with a super-polynomial convergence on the dimension by partly relying on the results in \cite{Suzuki15_2}. First we focus on the term
\begin{align*}
  \sum_{\substack{\bsk\in \NN_0^{s}\setminus \{\bszero\}\\ k_j<b^{dm}, \forall j}}\phi\left(b^{-\mu(\bsa;\bsk)}\right) .
\end{align*}
As mentioned in Subsection~\ref{subsec:result}, we use the purposely-designed concave function as in \eqref{eq:jensen_b_2} for $b=2$ and in \eqref{eq:jensen_b_geq_3} for $b\ge 3$. Note again that we set $\phi(0)=0$ for any $b$ and $0<\lambda\le 1$, and that the function $\phi$ is concave and unbounded monotonic increasing on $[0,\infty)$. We have the following result.

\begin{lemma}\label{lem:sum_phi_bound}
Let $\phi:[0,\infty) \to [0,\infty)$ be given as in (\ref{eq:jensen_b_2}) and (\ref{eq:jensen_b_geq_3}). If a sequence of the weights $\bsu$ satisfies $\liminf_{j\to \infty}\log(u_j^{-1})/j^r>0$ for some $r>0$, then there exists a constant $D'_{r,\lambda}$ depending only on $r$ and $\lambda$ such that
  \begin{align*}
    \sum_{\substack{\bsk\in \NN_0^{s}\setminus \{\bszero\}\\ k_j<b^{dm}, \forall j}}\phi\left(b^{-\mu(\bsa;\bsk)}\right) \le D'_{r,\lambda} ,
  \end{align*}
for any $(r+1)/(2r+1)<\lambda \le 1$.
\end{lemma}

\begin{proof}
Since the case $b=2$ can be proven in the same way as the case $b\ge 3$, we only consider the latter case in the following. Since $\phi(b^{-x})=b^{-x^\lambda}$ for $x\ge 1$, we have
  \begin{align}\label{eq:error_dependence_1}
    \sum_{\substack{\bsk\in \NN_0^{s}\setminus \{\bszero\}\\ k_j<b^{dm}, \forall j}}\phi\left(b^{-\mu(\bsa;\bsk)}\right) & \le \sum_{\substack{\bsk\in \NN_0^{s}\setminus \{\bszero\}\\ \mu(\bsa;\bsk)<1}}\phi\left(b^{-\mu(\bsa;\bsk)}\right)+\sum_{\substack{\bsk\in \NN_0^{s}\setminus \{\bszero\}\\ \mu(\bsa;\bsk)\ge 1}}\phi\left(b^{-\mu(\bsa;\bsk)}\right) \nonumber \\
    & = \sum_{\substack{\bsk\in \NN_0^{s}\setminus \{\bszero\}\\ \mu(\bsa;\bsk)<1}}\phi\left(b^{-\mu(\bsa;\bsk)}\right)+\sum_{\substack{\bsk\in \NN_0^{s}\setminus \{\bszero\}\\ \mu(\bsa;\bsk)\ge 1}}b^{-(\mu(\bsa;\bsk))^{\lambda}} \nonumber \\
    & = \sum_{\substack{\bsk\in \NN_0^{s}\setminus \{\bszero\}\\ \mu(\bsa;\bsk)<1}}\phi\left(b^{-\mu(\bsa;\bsk)}\right)+ \sum_{i=1}^{\infty}\sum_{\substack{\bsk\in \NN_0^{s}\setminus \{\bszero\}\\ i\le \mu(\bsa;\bsk)<i+1}}b^{-(\mu(\bsa;\bsk))^{\lambda}} \nonumber \\
    & \le \sum_{\substack{\bsk\in \NN_0^{s}\setminus \{\bszero\}\\ \mu(\bsa;\bsk)<1}}\phi\left(b^{-\mu(\bsa;\bsk)}\right) + \sum_{i=1}^{\infty}\vol_{\bsa}(i+1)b^{-i^{\lambda}} ,
  \end{align}
where 
  \begin{align*}
    \vol_{\bsa}(i+1) := \left| \{\bsk \in \NN_0^s \setminus \{\bszero\}: i\leq \mu(\bsa;\bsk)<i+1 \}\right| .
  \end{align*}

We now introduce the modified function $\mu'(\bsa; \cdot)$ as follows \cite[Definition~2.5]{Suzuki15_2}: For a real number $a$ and $k\in \NN$ whose $b$-adic expansion is given by $k = \kappa_1 b^{c_1-1}+\kappa_2 b^{c_2-1}+\cdots +\kappa_v b^{c_v-1}$ such that $\kappa_1,\dots,\kappa_v \in \{1,2,\dots,b-1\}$ and $c_1>\ldots > c_v >0$, we define
  \begin{align*}
    \mu'(a;k) := \sum_{i=1}^{v}\max(c_i+a ,1) ,
  \end{align*}
and $\mu'(a;0):=0$. For a vector $\bsa$ and $\bsk\in \NN_0^s$, we define
  \begin{align*}
    \mu'(\bsa;\bsk) := \sum_{j=1}^{s}\mu'(a_j;k_j) .
  \end{align*}
Then it holds from \cite[Section~2.2]{Suzuki15_2} that there exists a constant $c\ge 0$ such that $\mu(\bsa; \bsk) \ge \mu'(\bsa; \bsk) - c$ for any $\bsk\in \NN_0^s$. Thus we have $\vol_{\bsa}(i+1) \leq \vol_{\bsa}'(i+c+1)$ where
  \begin{align*}
    \vol_{\bsa}'(i+c+1) := \left| \{\bsk \in \NN_0^s \setminus \{\bszero\}: \mu'(\bsa;\bsk) <i+c+1 \}\right| .
  \end{align*}
  
Now the assumption $\liminf_{j\to \infty}\log(u_j^{-1})/j^r>0$ implies that there exist constants $a>0$ and $A \in \NN_0$ such that $a_j \ge aj^r$ holds for all $j > A$. Then it holds from \cite[Section~6.3]{Suzuki15_2} that the constant $c$ can be bounded above independently of $s$, and moreover from \cite[Lemma~6.12]{Suzuki15_2} that $\vol_{\bsa}'(i+c+1)$ is bounded by
  \begin{align*}
    \vol_{\bsa}'(i+c+1) \le \exp\left\{ A_{a,r}(i+c+1)^{(r+1)/(2r+1)}\right\} ,
  \end{align*}
where
  \begin{align*}
    A_{a,r} = (b-1)\left( A + \frac{\Gamma(1/r)}{ra^{1/r}}\right) + N + 1.
  \end{align*}
In the above, we write $N = (b-1)\sum_{j=1}^s|\{i \in \NN \mid i + a_j \leq 1\}|$, which is bounded by a constant independent of $s$ under the assumption $\liminf_{j\to \infty}\log(u_j^{-1})/j^r>0$, and $\Gamma(z)=\int_{0}^{\infty}t^{z-1}\exp(-t)\rd t$ denotes the Gamma function.

Thus, for the second term of (\ref{eq:error_dependence_1}), we obtain
  \begin{align*}
    \sum_{i=1}^{\infty}\vol_{\bsa}(i+1)b^{-i^{\lambda}} \le \sum_{i=1}^{\infty}\exp\left\{ A_{a,r}(i+c+1)^{(r+1)/(2r+1)}-i^{\lambda}\log b\right\}.
  \end{align*}
The above infinite sum converges when $\lambda > (r+1)/(2r+1)$, which implies that the second term of (\ref{eq:error_dependence_1}) is bounded by a constant independent of $s$.

Again from \cite[Section~6.3]{Suzuki15_2}, when $\liminf_{j\to \infty}\log(u_j^{-1})/j^r>0$ it is possible to show that the number of $\bsk\in \NN_0^{s}\setminus \{\bszero\}$ such that $\mu(\bsa;\bsk)<1$ is bounded by a constant independent of $s$. Thus, the first term of (\ref{eq:error_dependence_1}) is bounded by a constant independent of $s$. Since both terms of (\ref{eq:error_dependence_1}) are finite and depend only on $r$ and $\lambda$, we have the result.
\end{proof}

\begin{remark}\label{rm:endpoint}
If $A_{a,r}< \log b$, $D'_{r,\lambda}$ is finite even when $\lambda = (r+1)/(2r+1)$, and thus, Lemma~\ref{lem:sum_phi_bound} and Corollary~\ref{cor:cbc_bound} also hold for the endpoint $\lambda=(r+1)/(2r+1)$. This corresponds to Theorem~\ref{thm:suzuki} (shown by Suzuki in \cite{Suzuki15_2}), while our approach is constructive. 
\end{remark}

Now we are ready to prove Corollary~\ref{cor:cbc_bound}.
\begin{proof}[Proof of Corollary~\ref{cor:cbc_bound}]
Since $D'_{r,\lambda}$ is independent of $s$, there exists an $m_0\in \NN$ independent of $s$ such that either $D'_{r,\lambda}/(b^{m_0}-1)\le 1/e$ for the case $b=2$, or $D'_{r,\lambda}/(b^{m_0}-1)\le 1/b$ for the case $b\ge 3$. For $m\ge m_0$, we obtain
\begin{align*}
B_{\bsu}(\bsq,p) & \le \phi^{-1}\left[\frac{1}{b^m-1}\sum_{\substack{\bsk\in \NN_0^{s}\setminus \{\bszero\}\\ k_j<b^{dm}, \forall j}}\phi\left(b^{-\tilde{\mu}(\bsa;\bsk)}\right)\right] \\
& \le \phi^{-1}\left[\frac{D'_{r,\lambda}}{b^m-1}\right] = b^{-\left(\log_b\left(\frac{b^m-1}{D'_{r,\lambda}}\right)\right)^{1/\lambda}} .
\end{align*}
Since $B_{\bsu}(\bsq,p)$ is obviously finite even for $m< m_0$ and can be bounded independently of $s$ from Proposition~\ref{prop:cbc_bound} and Lemma~\ref{lem:sum_phi_bound}, there exists constants $D''_{r,\lambda},E''_{r,\lambda}>0$ both independent of $s$ such that 
\begin{align*}
B_{\bsu}(\bsq,p) \le E''_{r,\lambda} b^{-\left(\log_b\left(\frac{b^m-1}{D''_{r,\lambda}}\right)\right)^{1/\lambda}} ,
\end{align*}
for $m< m_0$ and any $(r+1)/(2r+1)< \lambda \le 1$. As $D'_{r,\lambda}$, $D''_{r,\lambda}$ and $E''_{r,\lambda}$ are all independent of $s$, there exists constants $D_{r,\lambda},E_{r,\lambda}>0$ such that the first part of Corollary~\ref{cor:cbc_bound} holds.

Let us consider the term $C_{\bsu}$. Using the inequality $\log(1+x)\le x$ for $x>0$, $C_{\bsu}$ can be bounded by
\begin{align*}
  \log(C_{\bsu}) & = \sum_{j=1}^{s}\sum_{h=1}^{d}\sum_{i=m+1}^{\infty}\log\left( 1+\frac{b-1}{b^{d(i-1)+h+a_j}}\right)\\
  & \le \sum_{j=1}^{s}\sum_{h=1}^{d}\sum_{i=m+1}^{\infty}\frac{b-1}{b^{d(i-1)+h+a_j}}= \frac{1}{b^{dm}}\sum_{j=1}^{s}\frac{1}{b^{a_j}} .
\end{align*}
Since the function $\exp(\cdot)$ is convex, we have
\begin{align}\label{bound_Cu}
  C_{\bsu}-1 & \le \exp\left(\frac{1}{b^{dm}}\sum_{j=1}^{s}\frac{1}{b^{a_j}}\right)-\exp(0) \le \frac{1}{b^{dm}}\left\{\exp\left(\sum_{j=1}^{s}\frac{1}{b^{a_j}}\right)-1\right\} .
\end{align}
As in the proof of Lemma~\ref{lem:sum_phi_bound}, the assumption $\liminf_{j\to \infty}\log(u_j^{-1})/j^r>0$ implies that there exist constants $a>0$ and $A \in \NN_0$ such that $a_j \ge aj^r$ holds for all $j > A$. Thus, we have
\begin{align*}
  \exp \left(\sum_{j=1}^{s}\frac{1}{b^{a_j}}\right) & = \exp \left(\sum_{j=1}^{A}\frac{1}{b^{a_j}}+\sum_{j=A+1}^{s}\frac{1}{b^{a_j}}\right) \\
  & \le \exp \left\{\sum_{j=1}^{A}\left( \frac{1}{b^{a_j}}-\frac{1}{b^{aj^r}}\right)+\sum_{j=1}^{\infty}\frac{1}{b^{aj^r}}\right\} \\
  & \le \exp \left\{ A\max_{j=1,\ldots,A}\left( \frac{1}{b^{a_j}}-\frac{1}{b^{aj^r}}\right)+\frac{\Gamma(1/r)}{r(a\log b)^{1/r}}\right\} =: E''_r,
\end{align*}
where we use the result of \cite[Lemma~6.11]{Suzuki15_2} in the second inequality and $\Gamma$ again denotes the Gamma function. Hence we have 
\begin{align*}
  C_{\bsu}-1 \le \frac{E''_r}{b^{dm}}\quad \text{and}\quad C_{\bsu}\le E''_r .
\end{align*}
Now let $d\ge m^{r/(r+1)}$. Using the above bounds on $B_{\bsu}(\bsq,p)$ and $C_{\bsu}$, we have
\begin{align*}
  e^{\wor}(\Fcal_{s,\bsu};\Dcal_d(P(\bsq,p))) & \le C_{\bsu}-1+C_{\bsu}B_{\bsu}(\bsq,p) \\
  & \le E''_r\left\{b^{-m^{(2r+1)/(r+1)}}+E_{r,\lambda}b^{-\left(\log_b\left(\frac{b^m-1}{D_{r,\lambda}}\right)\right)^{1/\lambda}} \right\} \\
  & \le E'_{r,\lambda}b^{-\left(\log_b\left(\frac{b^m-1}{D_{r,\lambda}+1}\right)\right)^{1/\lambda}} ,
\end{align*}
where $E'_{r,\lambda}=E''_r(1+E_{r,\lambda})$, which proves the second part of Corollary~\ref{cor:cbc_bound}.
\end{proof}

%%%%%%%%%%%%%%%%%%%%%%%%%%%%%%%%%%%%%%%%%%%%%%%%%%%%%%%%%%%
\subsection{Fast construction algorithm}\label{subsec:fast_cbc}
We show how one can apply the fast CBC construction of \cite{NC06, NC06_2} using the fast Fourier transform. According to Algorithm~\ref{algorithm:cbc}, we choose an irreducible polynomial $p$ with $\deg(p)=m$, set $q_1=1$ and construct the polynomials $q_2,q_3,\ldots,q_{ds}$ inductively. For some $1\le \tau<ds$, assume that $\bsq_{\tau-1}=(q_1,\ldots,q_{\tau-1})$ are already found. Let $\beta = \lfloor \tau /d\rfloor$ and $\gamma=\tau-d(\beta-1)$. Using the notation as in the proof of Proposition~\ref{prop:worst-case_error_interlaced2}, we compute
\begin{align*}
  & \quad B_{\bsu}((\bsq_{\tau-1},q),p) \\
& = -1+\frac{1}{b^m}\sum_{n=0}^{b^m-1}\left(\prod_{j=1}^{\beta-1}\prod_{h=1}^{d}\psi_{a_j,h,m}(x_{n,d(j-1)+h})\right) \prod_{h=1}^{\gamma}\psi_{a_{\beta},h,m}(x_{n,d(\beta-1)+h})  ,
\end{align*}
for all $q\in R_{b,m}$, where $\{\bsx_0,\ldots,\bsx_{b^m-1}\}\subset [0,1)^{\tau}$ is a polynomial lattice point set $P((\bsq_{\tau-1},q),p)$. 

We introduce the following notation
\begin{align*}
  \rho_{n,\tau-1}:=\left(\prod_{j=1}^{\beta-1}\prod_{h=1}^{d}\psi_{a_j,h,m}(x_{n,d(j-1)+h})\right) \prod_{h=1}^{\gamma-1}\psi_{a_{\beta},h,m}(x_{n,d(\beta-1)+h})  ,
\end{align*}
for $0\le n<b^m$, where the empty product is equal to $1$. Then it is straightforward to confirm that 
\begin{align*}
  B_{\bsu}((\bsq_{\tau-1},q),p) & = -1+\frac{1}{b^m}\left[ \rho_{0,\tau-1}\psi_{a_{\beta},\gamma,m}(0) +\sum_{n=1}^{b^m-1}\rho_{n,\tau-1}\psi_{a_{\beta},\gamma,m}(x_{n,\tau}) \right]
\end{align*}
holds. Therefore, in order to find $q=q_{\tau}\in R_{b,m}$ which minimizes $B_{\bsu}((\bsq_{\tau-1},q),p)$ as a function of $q$, we only need to compute
\begin{align}\label{eq:fast_cbc_1}
  \sum_{n=1}^{b^m-1}\rho_{n,\tau-1}\psi_{a_{\beta},\gamma,m}(x_{n,\tau}) 
\end{align}
for all $q\in R_{b,m}$. In the following, we show how we can exploit a feature of polynomial lattice point sets to apply the fast CBC construction using the fast Fourier transform.

Since we choose an irreducible polynomial $p$, there exists a primitive element $g\in R_{b,m}$, which satisfies
\begin{align*}
     \{g^0 \bmod{p}, g^1 \bmod{p}, \ldots, g^{b^m-2} \bmod{p}\} = R_{b,m},
\end{align*}
and $g^{-1} \bmod{p}=g^{b^m-2} \bmod{p}$. When $q=g^i \bmod{p}$, we can rewrite (\ref{eq:fast_cbc_1}) as
\begin{align*}
   c_i & = \sum_{n=0}^{b^m-2}\rho_{g^{-n} \bmod{p},\tau-1}\psi_{a_{\beta},\gamma,m}\left( x_{g^{-n}\bmod{p},\tau}\right) \\
   & = \sum_{n=0}^{b^m-2}\rho_{g^{-n} \bmod{p},\tau-1}\psi_{a_{\beta},\gamma,m}\left( v_m\left( \frac{(g^{i-n}\bmod p)(x)}{p(x)}\right)\right) ,
\end{align*}
for $0\le i<b^m-1$. Here the first argument $g^{-n} \bmod{p}$ of $\rho$ is understood as an integer by identifying the polynomial $g^{-n} \bmod{p}\in \ZZ_b[x]$ with an integer based on its $b$-adic expansion. Let us denote $\bsc=(c_i)_{0\le i< b^m-1}$, $\bsrho_{\tau-1}=(\rho_{g^{-n} \bmod{p},\tau-1})_{0\le n< b^m-1}$, and 
\begin{align*}
   \Psi_{a_\beta,\gamma} = \left[ \psi_{a_{\beta},\gamma,m}\left( v_m\left( \frac{(g^{i-n}\bmod p)(x)}{p(x)}\right)\right)\right] _{0\le i,n<b^m-1} .
\end{align*}
Then we have
\begin{align*}
   \bsc = \Psi_{a_\beta,\gamma}\bsrho_{\tau-1} .
\end{align*}
Let $i_0$ be an integer $0\le i_0 < b^m-1$ which satisfies $c_{i_0}\le c_{i}$ for all $0\le i<b^m-1$. Then we set $q_{\tau}=g^{i_0} \bmod{p}$. After finding $q_{\tau}$, we just update $\bsrho_{\tau}$ by
\begin{align*}
  \rho_{g^{-n} \bmod{p},\tau}:=\rho_{g^{-n} \bmod{p},\tau-1} \psi_{a_{\beta},\gamma,m}\left( v_m\left( \frac{(g^{i_0-n}\bmod p)(x)}{p(x)}\right)\right) .
\end{align*}

Since the matrix $\Psi_{a_\beta,\gamma}$ is circulant, we only need to evaluate one column (or one row) of $\Psi_{a_\beta,\gamma}$ for computing all the elements of $\Psi_{a_\beta,\gamma}$. One column consists of $b^m-1$ elements, each of which can be evaluated in $O(m)$ arithmetic operations. Moreover, the matrix-vector multiplication $\Psi_{a_\beta,\gamma}\bsrho_{\tau-1}$ can be efficiently done in $O(mb^m)$ arithmetic operations by using the fast Fourier transform as shown in \cite{NC06,NC06_2}. This reduces the computational cost significantly as compared to the naive matrix-vector multiplication. Since we construct the polynomials $q_2,q_3,\ldots,q_{ds}$ inductively, the total computational cost becomes $O(dsmb^m)$ arithmetic operations. As for memory, we only need to store the vector $\bsrho_{\tau}$, which requires $O(b^m)$ memory space.

From Corollary~\ref{cor:cbc_bound}, to obtain a worst-case error bound with a dimension-independent super-polynomial convergence, it is sufficient to set $d\ge m^{r/(r+1)}$. This means that the total computational cost is given by $O(sm^{(2r+1)/(r+1)}b^m)$ arithmetic operations, which can be bounded above by $O(sm^2b^m) (= O(sN(\log N)^2))$ arithmetic operations for any $r>0$.
%%%%%%%%%%%%%%%%%%%%%%%%%%%%%%%%%%%%%%%%%%%%%%%%%%%%%%%%%%%
%%%%%%%%%%%%%%%%%%%%%%%%%%%%%%%%%%%%%%%%%%%%%%%%%%%%%%%%%%%
%%%%%%%%%%%%%%%%%%%%%%%%%%%%%%%%%%%%%%%%%%%%%%%%%%%%%%%%%%%
\section{Numerical experiments}\label{sec:exp}
We conclude this paper with numerical experiments. In our experiments, we focus on the case $b=2$ and choose the weights $u_j=2^{-j^r}$ for $r>0$, i.e., $a_j=j^r$, which satisfies the condition $\liminf_{j\to \infty}\log(u_j^{-1})/j^r>0$. In Figure~\ref{fig:1}, we report the values of $B_{\bsu}(\bsq,p)$ as functions of $m\ (=\deg(p))$ up to 15 with several values of $s$ for the cases $r=0.5$ (top), $r=1$ (middle), and $r=2$ (bottom), respectively, where $p$ and $\bsq$ are found by Algorithm~\ref{algorithm:cbc}, and the interlacing factor is given by $d=\lceil m^{r/(r+1)} \rceil$, as Corollary~\ref{cor:cbc_bound} suggests. Since the first $dm$ digits appearing in the dyadic expansion of every coordinate of each point can be non-zero, the range of $m$ is restricted so that $dm\leq 52$, by considering the double precision arithmetic. For $r=0.5$ and $r=1$, one can observe that the rate of convergence is improved as $m$ increases for the low-dimensional cases $s=1$ and $s=2$, while one cannot do so for the higher-dimensional cases within this range of $m$. For $r=2$, one can observe an improvement of the rate of convergence even for the higher-dimensional cases and the convergence behaviors for $s\geq 2$ are almost identical, i.e., dimension-independent.

Next we consider the test function 
\begin{align*}
  f_1(\bsx) = \prod_{j=1}^{s}\exp\left( -\frac{x_j}{2^{j^r}}\right)
\end{align*}
as integrand, which belongs to $\Fcal_{s,\bsu}$ for $u_j=2^{-j^r}$. Figure~\ref{fig:2} shows the absolute integration error $|I(f_1;\Dcal_d(P(\bsq,p)))-I(f_1)|$ as functions of $m$ with several values of $s$ for $r=0.5$ (top), $r=1$ (middle), and $r=2$ (bottom), respectively, where $p$ and $\bsq$ are again found by Algorithm~\ref{algorithm:cbc} for each given $r$. Regardless of $r$, one can observe that a convergence better than $1/N$ is achieved for any $s$. For $r=0.5$, we see a sudden drop between $m=8$ and $m=9$ for the low-dimensional cases $s=1$ and $s=2$. This is because of the increment of the interlacing factor $d$ from $2$ to $3$. A similar behavior can be found for $r=1$ between $m=9$ and $m=10$, where the interlacing factor $d$ is incremented from $3$ to $4$. It is interesting to see that the rate of convergence is improved after $d$ increases. For $r=2$, the rate of convergence is roughly $N^{-4}$ even for $s=16$ and the convergence behaviors for $s\geq 2$ are almost identical.

Finally we consider the following two test functions
\begin{align*}
f_2(\bsx) & = \prod_{j=1}^{s}\left( 1+\frac{w^{j}}{21}\left( -10+42x_j^2-42x_j^5+21x_j^6\right)\right) \; \text{for $w>0$}, \\
f_3(\bsx) & = \prod_{j=1}^{s}\Big( 1+\frac{w^{j}}{8}( 31-84x_j^2+8x_j^3+70x_j^4-28x_j^6+8x_j^7 \\
          & \quad \qquad  -16\cos(1)-16\sin(x_j))\Big) \; \text{for $w>0$}.
\end{align*}
These were used in \cite{DNP14} as examples belonging to a certain function class with finite smoothness. In fact $f_3\notin \Fcal_{s,\bsu}$ for any choice of $\bsu$ with $u_j=2^{-j^r}$ for $r>0$, and thus, there is no theoretical guarantee that our constructed interlaced polynomial lattice rule does work efficiently. We compare the performance of our constructed rule with that of the Sobol' sequence. We search for $p$ and $\bsq$ by setting $r=1$ in Algorithm~\ref{algorithm:cbc}. Figures~\ref{fig:3} and \ref{fig:4} show the absolute integration error as functions of $m$ for $f_2$ and $f_3$, respectively. We see that our constructed rule is superior to Sobol' sequence in terms of both the magnitude of the error and the rate of convergence. Regardless of the integrand and the dimension, the Sobol' sequence achieves an error convergence of order $N^{-1}$. On the other hand, our constructed rule can exploit the smoothness of the functions and achieve an error convergence of higher order. Thus our constructed rule works even for some functions not belonging to $\Fcal_{s,\bsu}$.

\section*{Acknowledgements}

The research of J. Dick, K. Suzuki and T. Yoshiki was supported under the Australian Research Councils Discovery Projects funding scheme (project number DP150101770).
The research of T. Goda was supported by JSPS Grant-in-Aid for Young Scientists No.15K20964.

%%%%%%%%%%%%%%%%%%%%%%%%%%%%%%%%%%%%%%%%%%%%%%%%%%%%%%%%%%%
%%%%%%%%%%%%%%%%%%%%%%%%%%%%%%%%%%%%%%%%%%%%%%%%%%%%%%%%%%%
%%%%%%%%%%%%%%%%%%%%%%%%%%%%%%%%%%%%%%%%%%%%%%%%%%%%%%%%%%%

\begin{figure}
\begin{center}
\includegraphics[width=7cm]{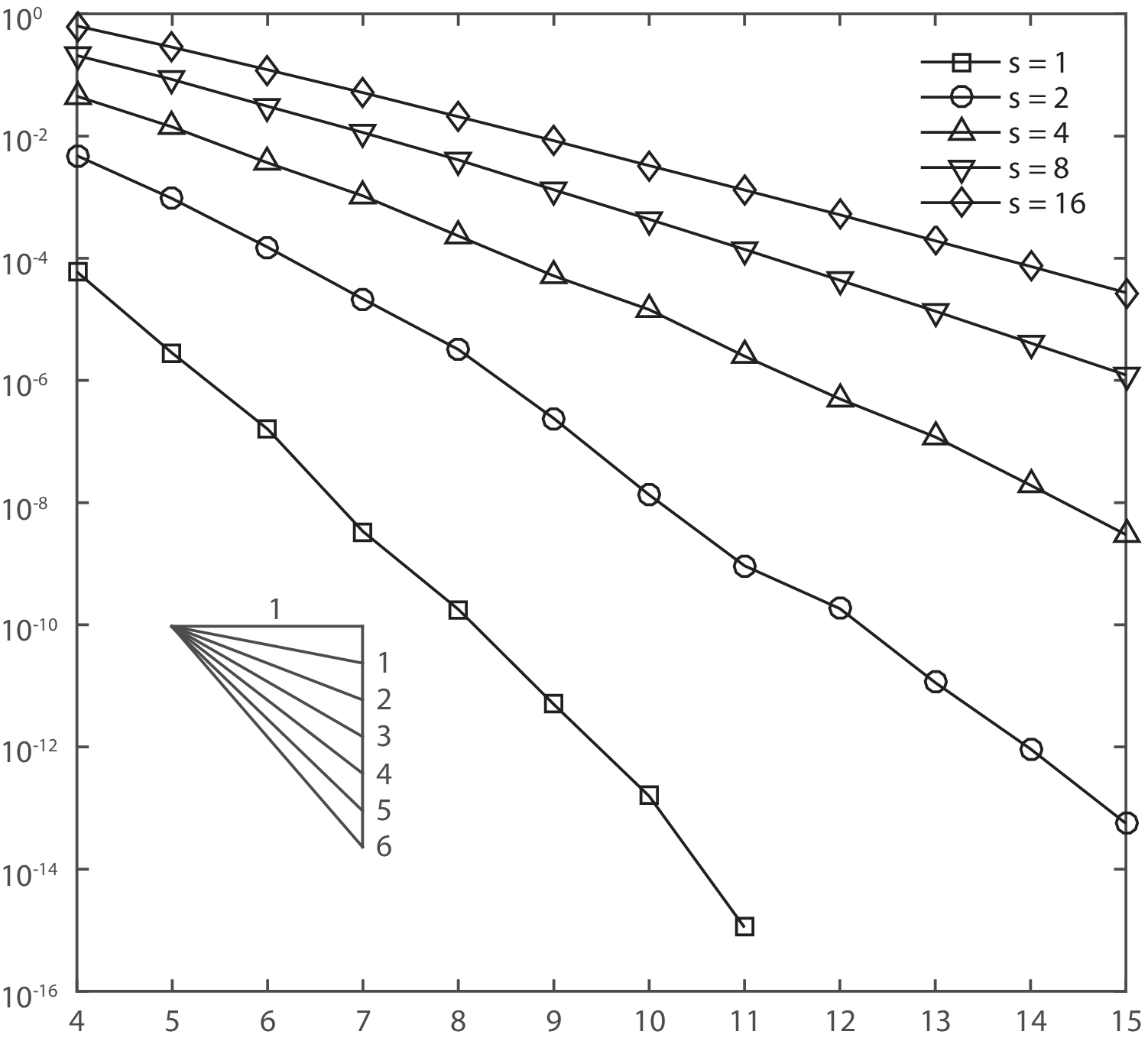}
\includegraphics[width=7cm]{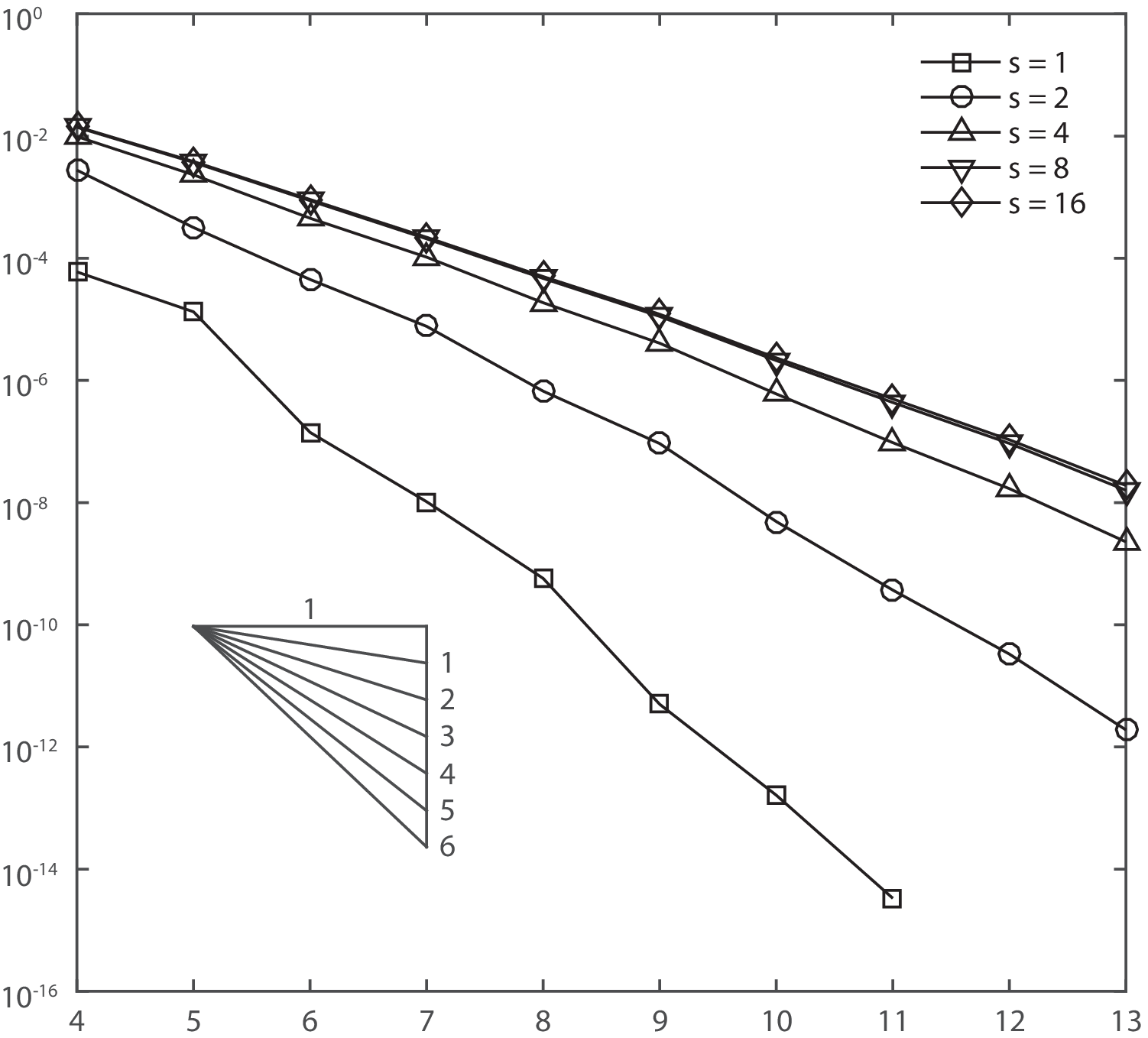}
\includegraphics[width=7cm]{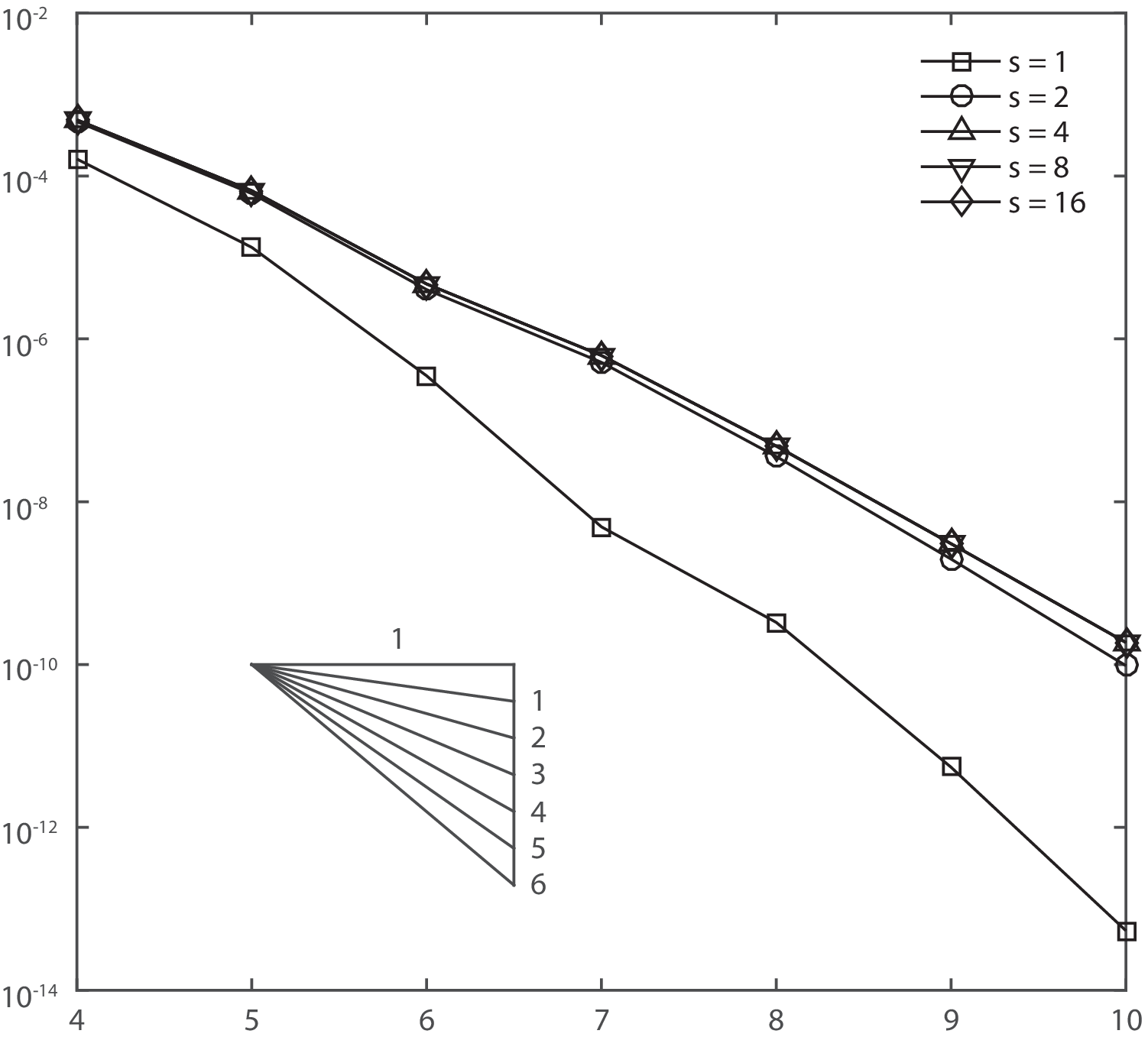}
\caption{$B_{\bsu}(\bsq,p)$ as functions of $m$ for $r=0.5$ (top), $r=1$ (middle), and $r=2$ (bottom).}
\label{fig:1}
\end{center}
\end{figure}
\begin{figure}
\begin{center}
\includegraphics[width=7cm]{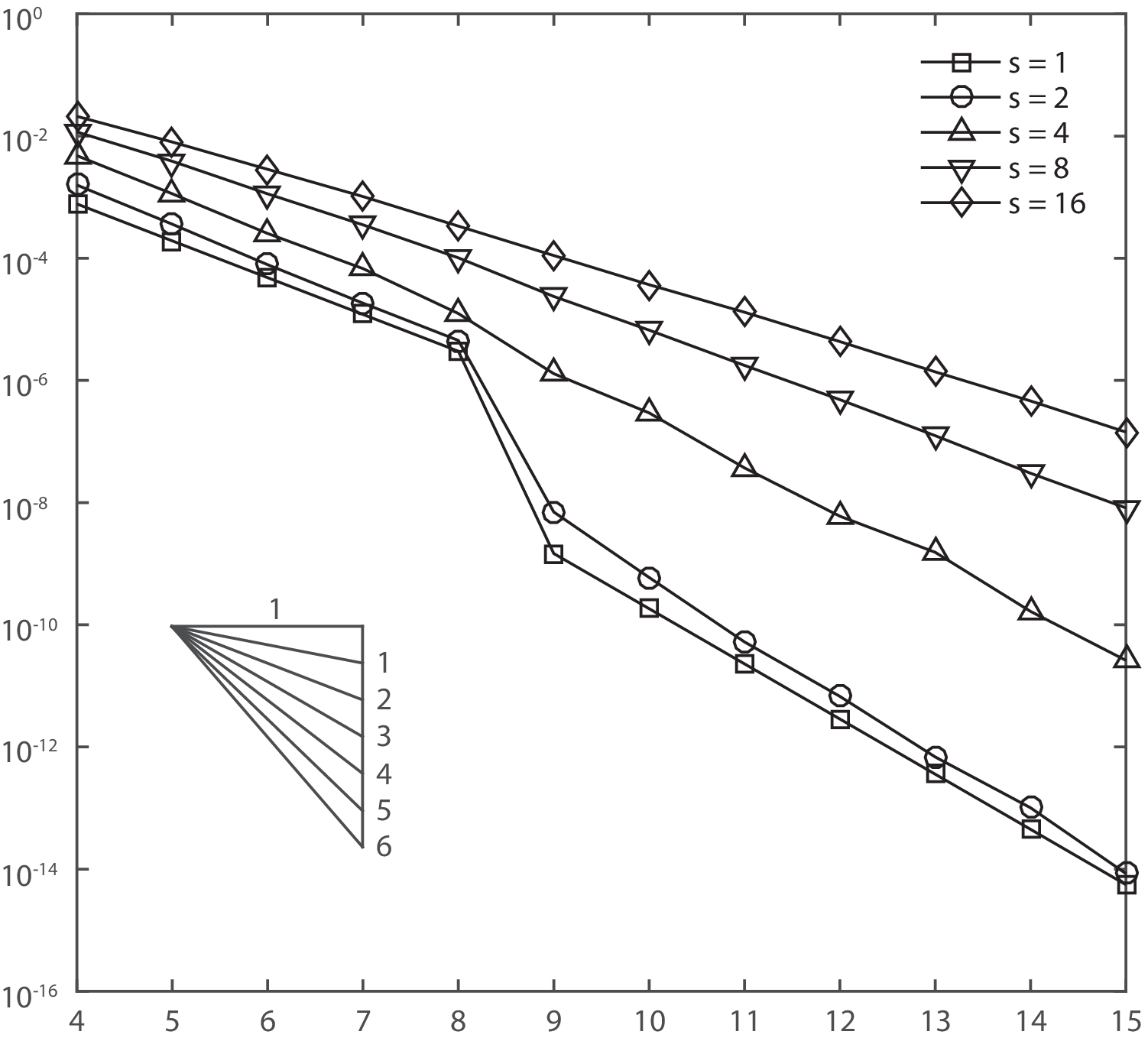}
\includegraphics[width=7cm]{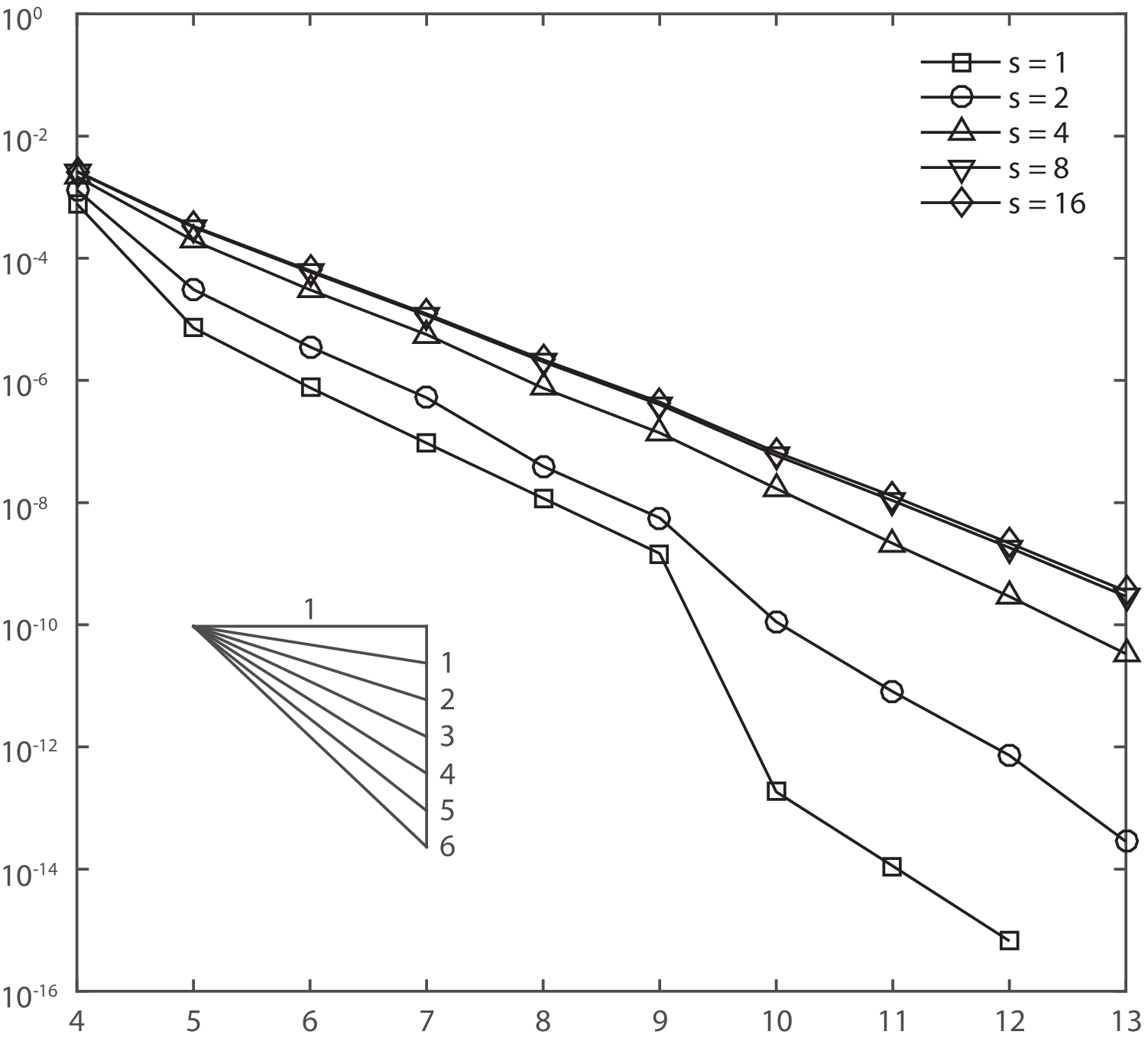}
\includegraphics[width=7cm]{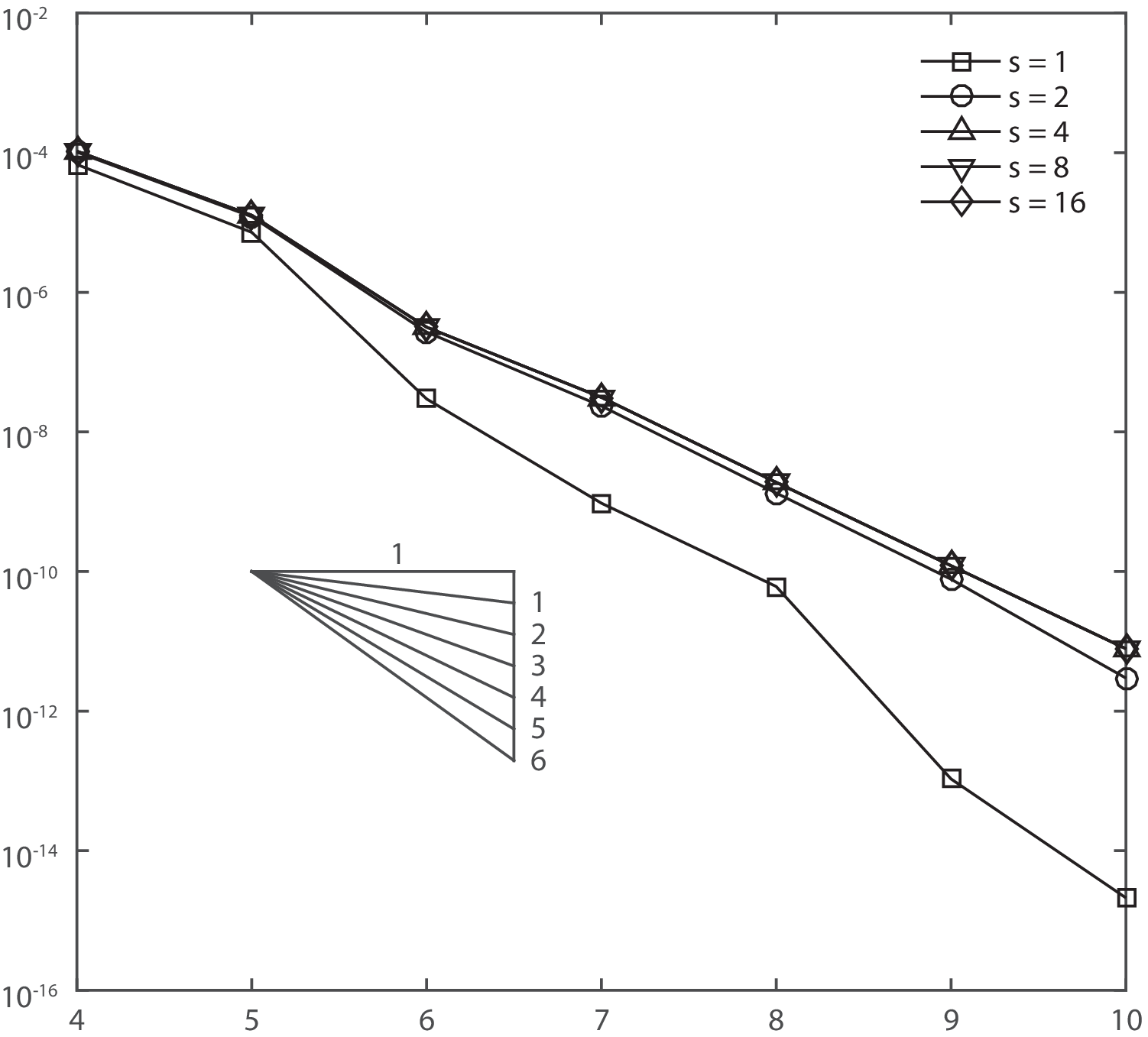}
\caption{The absolute integration error as functions of $m$ for $f_1$ with $r=0.5$ (top), $r=1$ (middle), and $r=2$ (bottom).}
\label{fig:2}
\end{center}
\end{figure}
\begin{figure}
\begin{center}
\includegraphics[width=6cm]{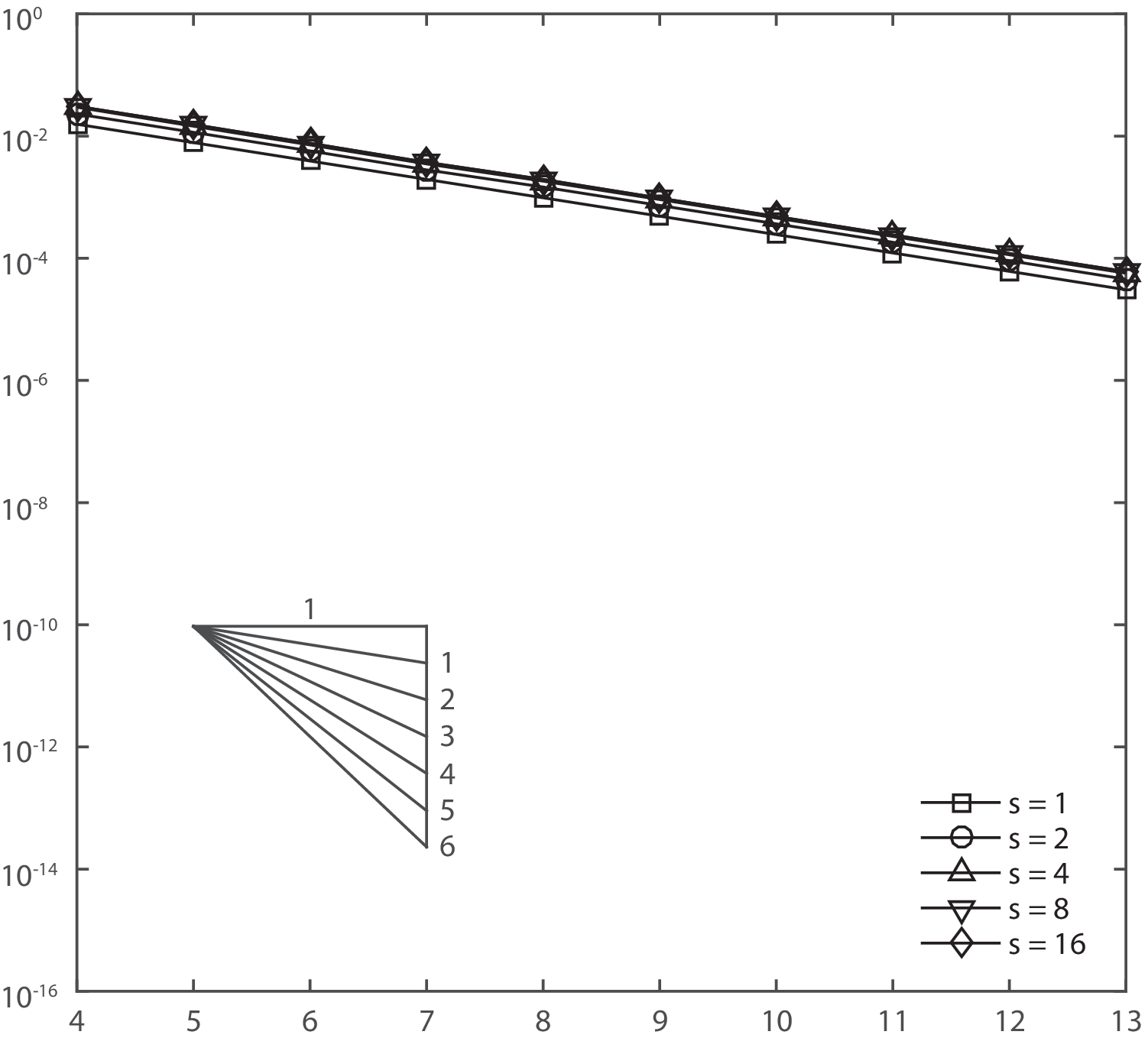}
\includegraphics[width=6cm]{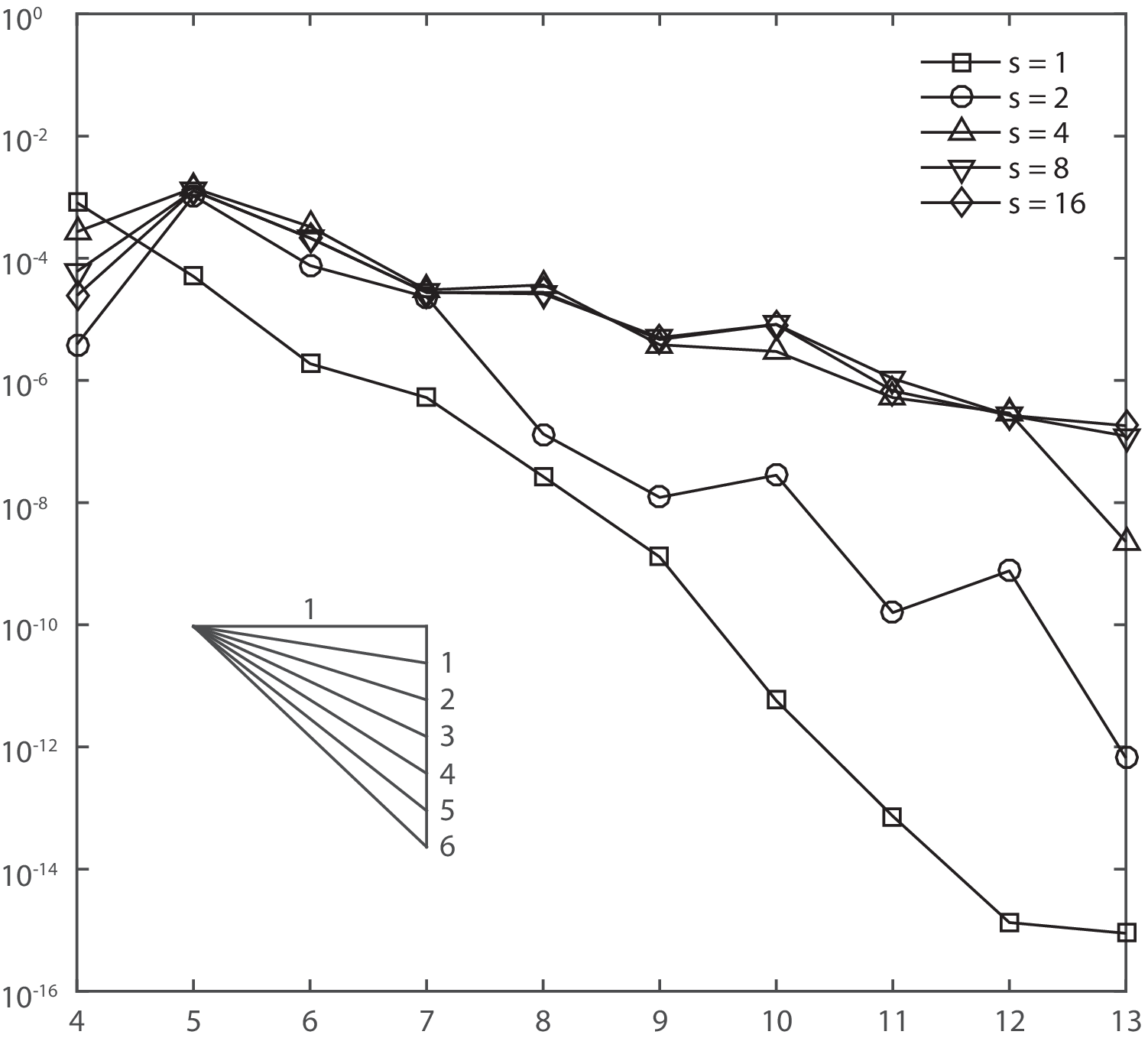}\\
\includegraphics[width=6cm]{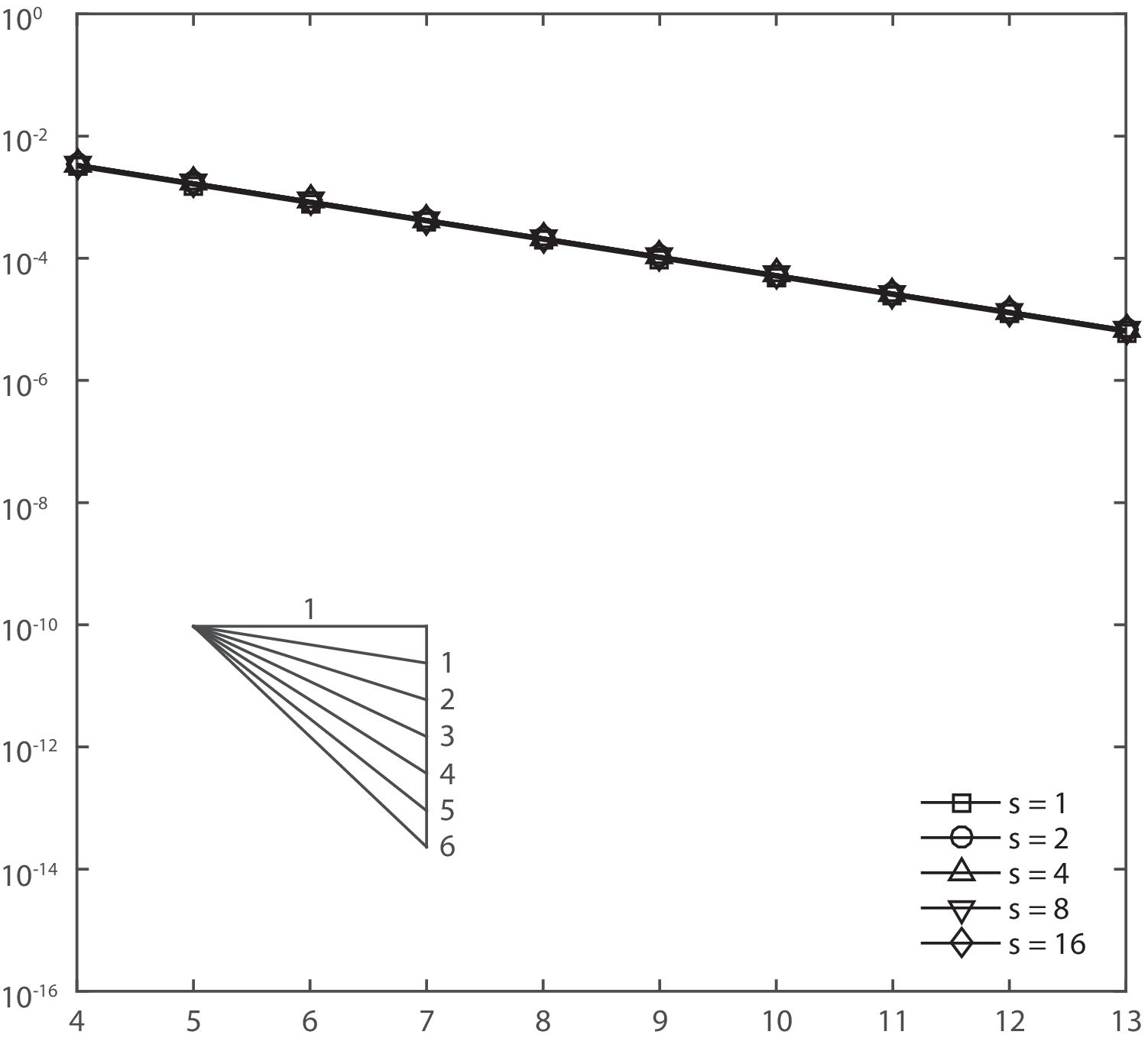}
\includegraphics[width=6cm]{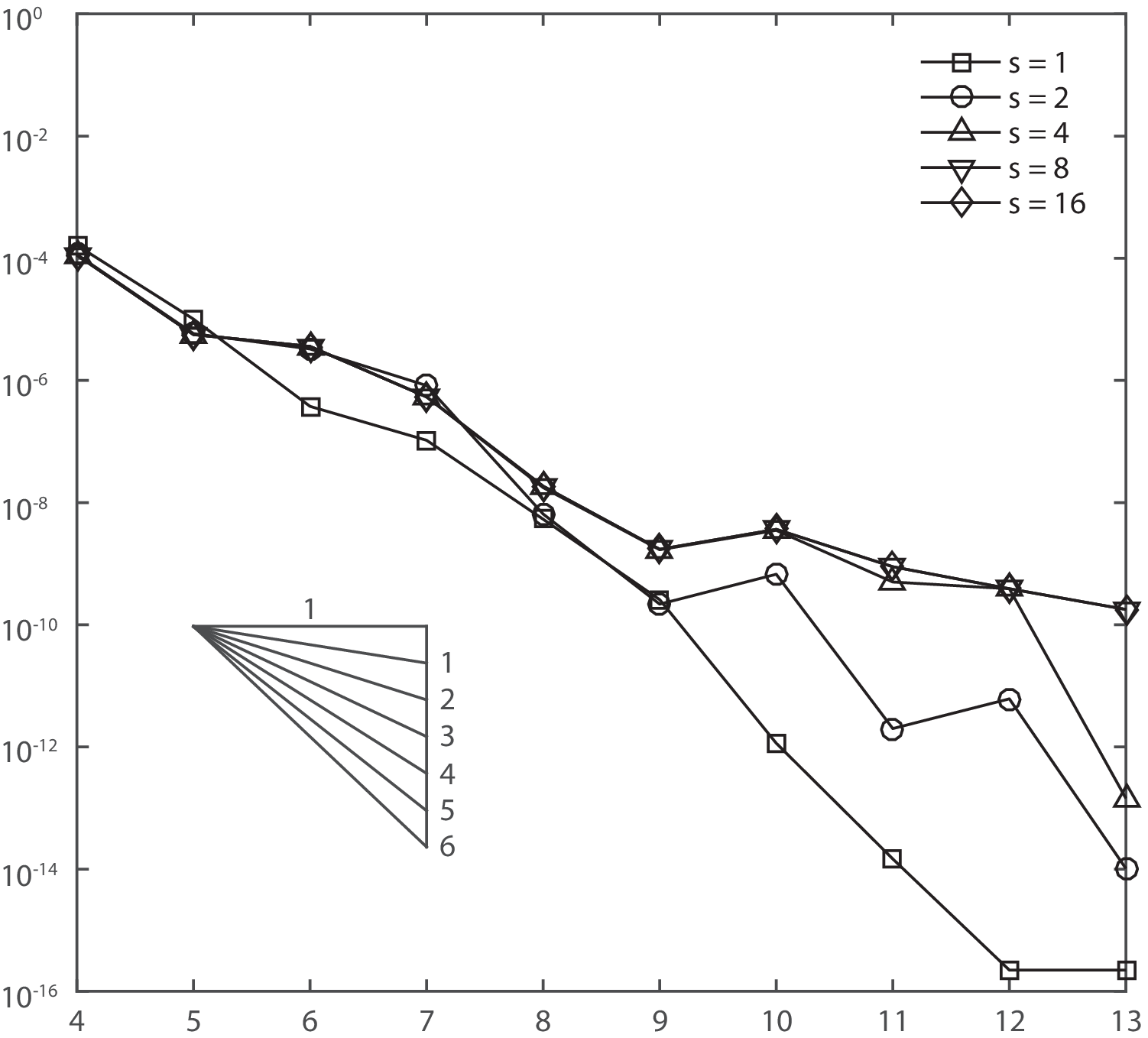}
\caption{The absolute integration error as functions of $m$ for $f_2$ with $w=0.5$ (top) and $w=0.1$ (bottom) by Sobol' sequence (left) and our constructed interlaced polynomial lattice rule (right).}
\label{fig:3}
\end{center}
\end{figure}
\begin{figure}
\begin{center}
\includegraphics[width=6cm]{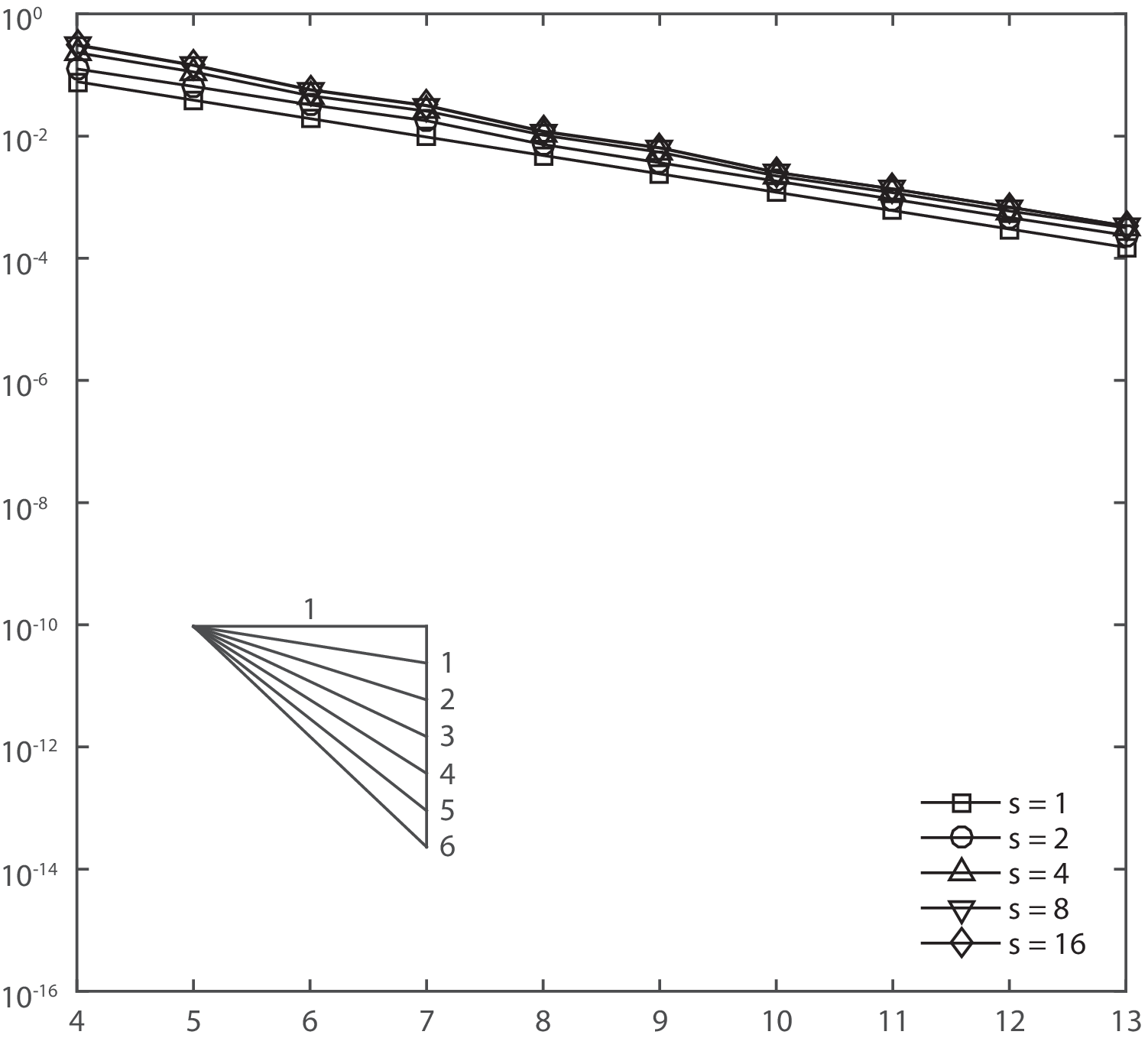}
\includegraphics[width=6cm]{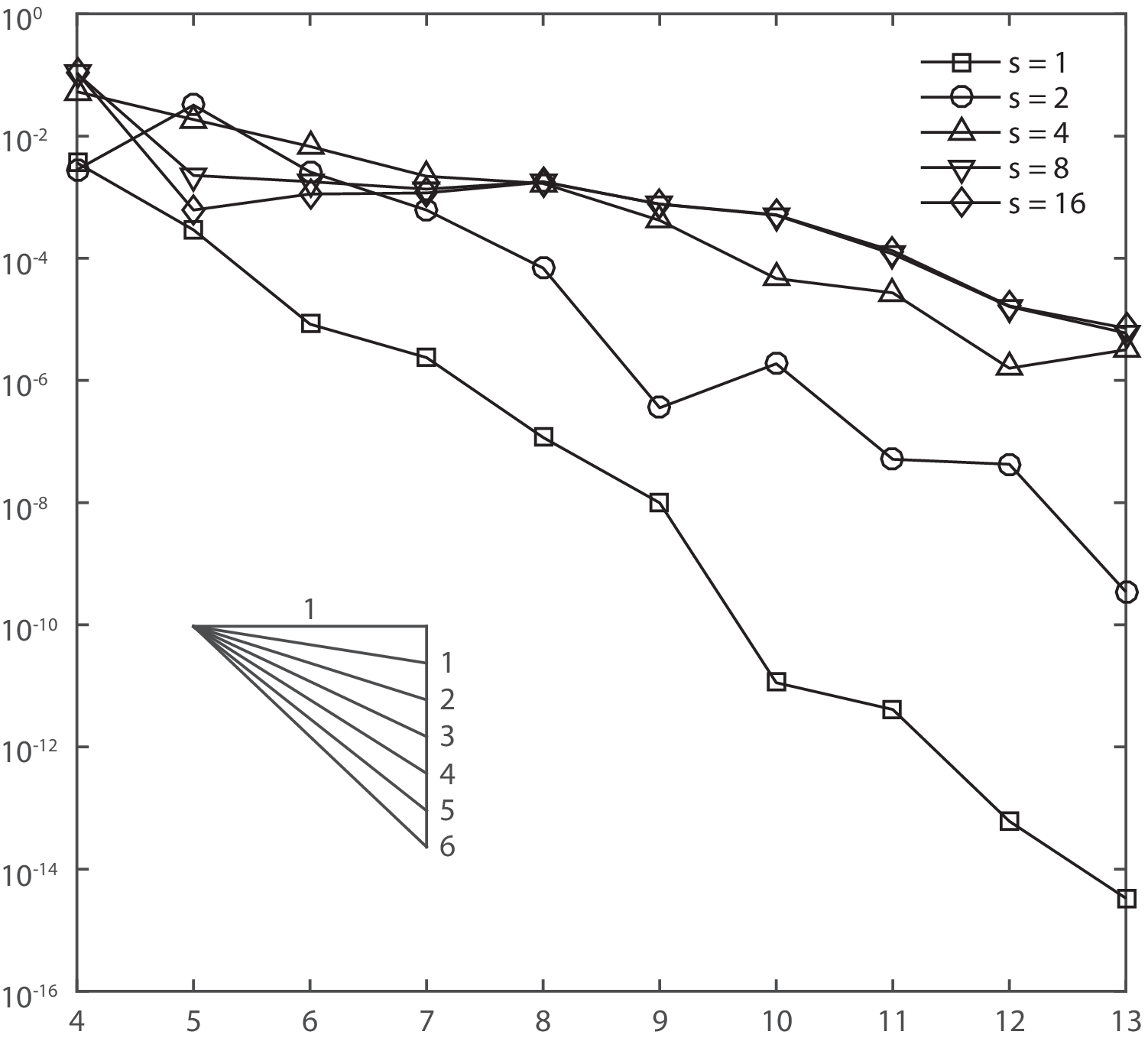}\\
\includegraphics[width=6cm]{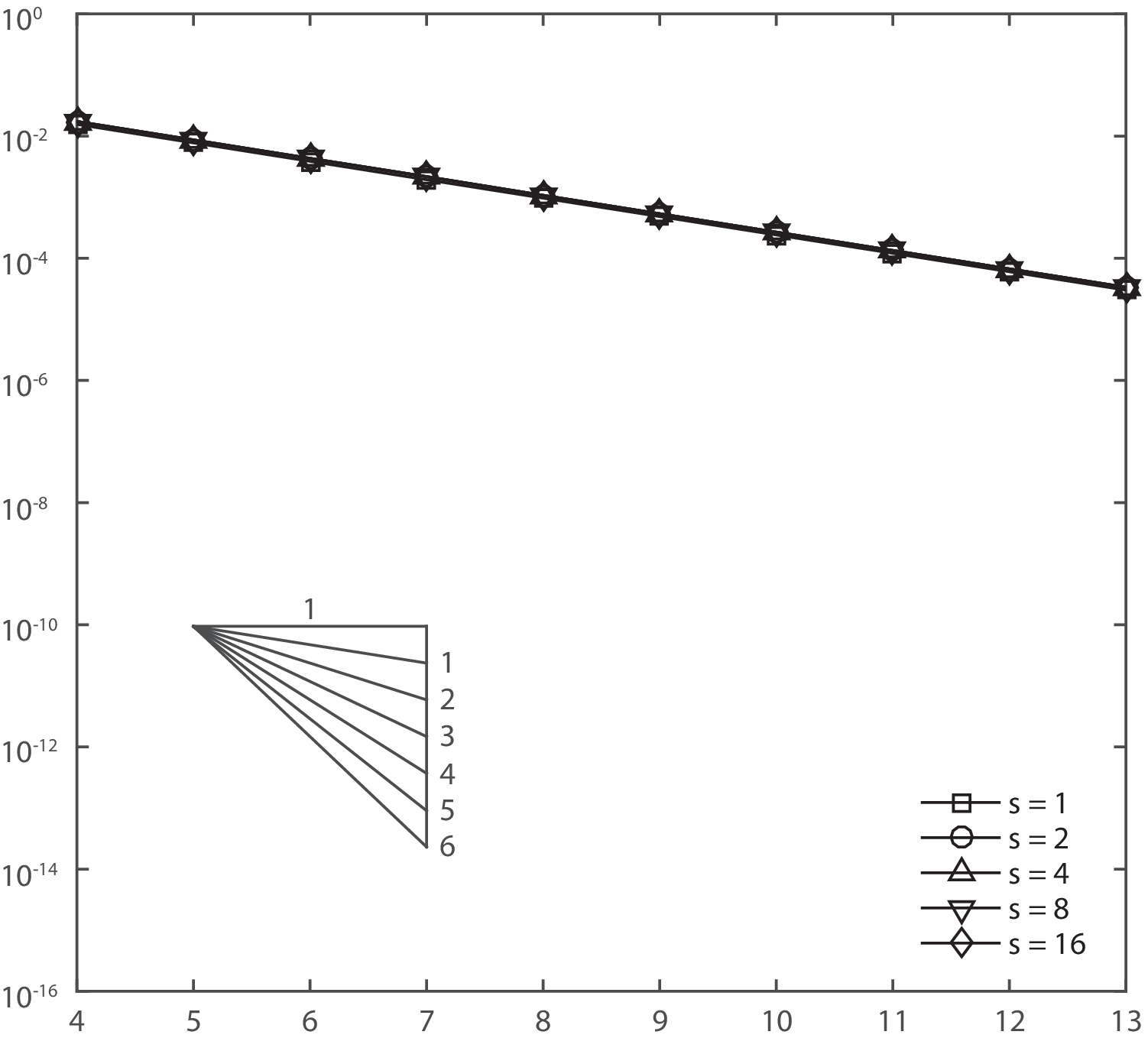}
\includegraphics[width=6cm]{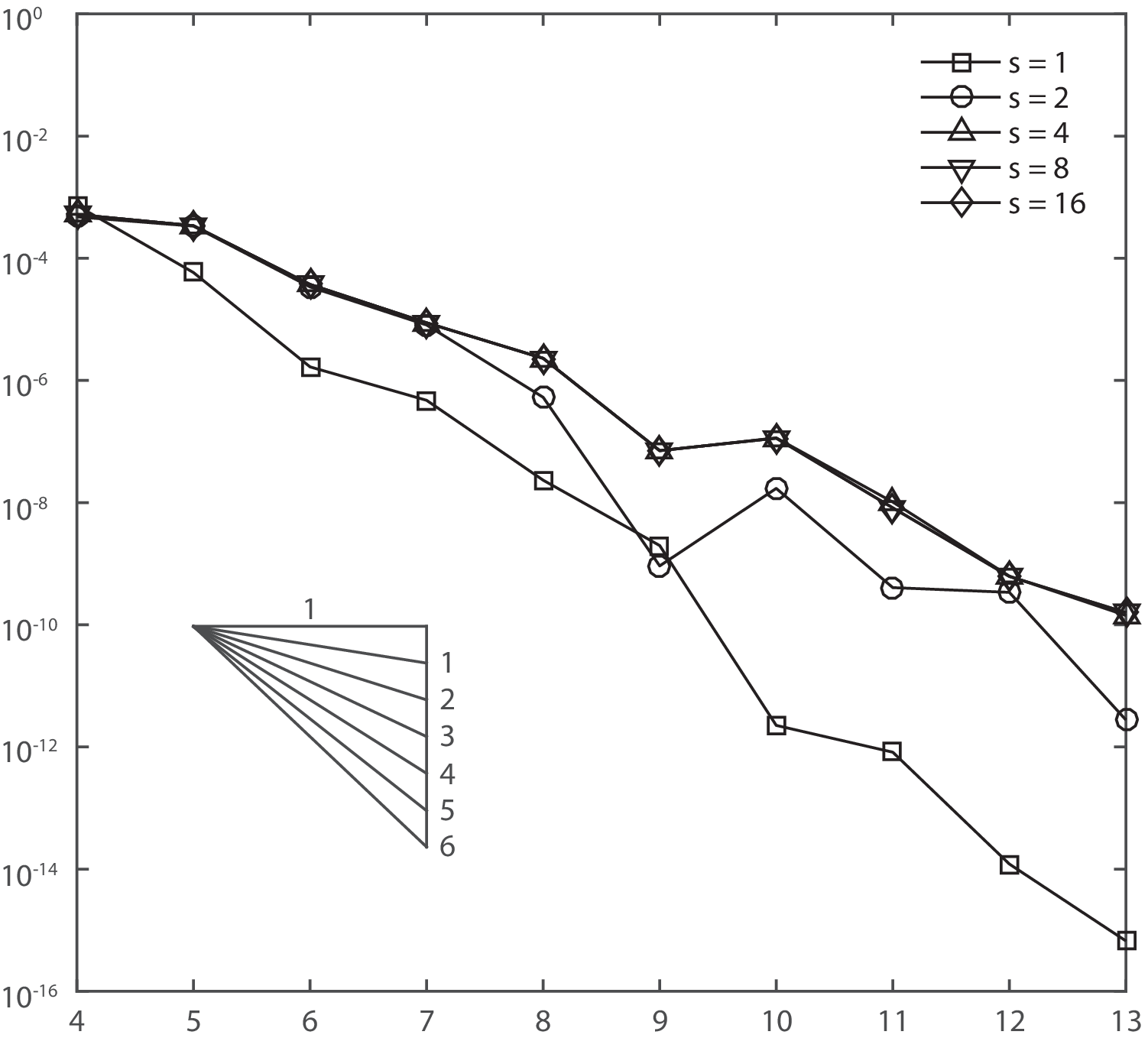}
\caption{The absolute integration error as functions of $m$ for $f_3$ with $w=0.5$ (top) and $w=0.1$ (bottom) by Sobol' sequence (left) and our constructed interlaced polynomial lattice rule (right).}
\label{fig:4}
\end{center}
\end{figure}

\end{document}